\documentclass[12pt]{amsart}
\usepackage{color,epsfig,url,hhline,graphicx%,showkeys
}
\usepackage[dvipsnames]{xcolor}
\numberwithin{figure}{section}
\numberwithin{table}{section}
\newtheorem{theorem}{Theorem}[section]

\newtheorem{lemma}[theorem]{Lemma}
\newtheorem{prop}[theorem]{Proposition}

\theoremstyle{definition}
\newtheorem{definition}[theorem]{Definition}
\newtheorem{example}[theorem]{Example}
\newtheorem{notation}[theorem]{Notation}
\newtheorem{cor}[theorem]{Corollary}

\theoremstyle{remark}
\newtheorem{remark}[theorem]{Remark}

\numberwithin{equation}{section}
\newfont{\tap}{tap scaled 650}

\def \h{{\mathfrak h}}
\def \Z{{\mathbb Z}}
\def \S{{\mathbb S}}
\def \O{{\mathcal O}}

\def \[{[ }
\def \]{] }

\def\t{\widetilde}

\def\h{\widehat}

\def\S{\Sigma}

\newcommand{\bm}{\boldsymbol}
\def\b{\bm{b}}

\def \sgn{\mathrm{sgn }}

\definecolor{dgreen}{rgb}{0,0.5,0}        
\definecolor{dred}{rgb}{0.5,0,0}

\begin{document}

\title[Cluster algebras of finite mutation type with coefficients]{ Cluster algebras of finite mutation type \\ with coefficients}
\author{Anna Felikson,  Pavel Tumarkin}
\address{Department of Mathematical Sciences, Durham University, Mathematical Sciences \& Computer Science Building, Upper Mountjoy Campus, Stockton Road, Durham, DH1 3LE, UK}
\email{anna.felikson@durham.ac.uk, pavel.tumarkin@durham.ac.uk}
\thanks{Research of P.T. was supported in part by the Leverhulme Trust research grant RPG-2019-153.}

\begin{abstract}
We classify mutation-finite cluster algebras with arbitrary coefficients of geometric type. This completes the classification of all mutation-finite cluster algebras started in~\cite{FeSTu1}. 
\end{abstract}

\maketitle
\setcounter{tocdepth}{1}
\tableofcontents

\section{Introduction and main results}

Cluster algebras with coefficients were introduced in~\cite{FZ4}, the fourth paper in the series founding the theory of cluster algebras. Cluster algebras of geometric type are defined as those having their coefficients in tropical semifields. In particular, this includes the important case of cluster algebras with principal coefficients.

A cluster algebra of geometric type is completely defined by an integer $(m+n)\times n$ {\em exchange matrix} with skew-symmetrizable top $n\times n$ part (called {\em principal} or {\em mutable} part). Exchange matrices undergo involutive transformations called {\em mutations}, all exchange matrices which can be obtained by iterative mutations form a {\em mutation class}. We say that a cluster algebra is {\em mutation-finite} if its mutation class is finite.

Coefficient-free mutation-finite cluster algebras were classified in~\cite{FeSTu1,FeSTu2}. These algebras found various applications, including ones in quantum field theories (see e.g.~\cite{CV,CV2}).

In this paper, we classify all mutation-finite exchange matrices with arbitrary coefficients. We first restrict ourselves to matrices with skew-symmetric mutable part (this assumption will be dropped later). In this case the matrix can be represented by a quiver with vertices of two types: {\em mutable} (corresponding to the mutable part of the matrix) and {\em frozen} (such quivers are also called {\em ice quivers}). The quiver also undergoes mutations compatible with mutations of the matrix, we say that a quiver is mutation-finite if the corresponding exchange matrix is. 

The first easy observation is that the mutable part of a mutation-finite quiver should be mutation-finite. Mutation-finite quivers without frozen vertices were classified in~\cite{FeSTu1}, the list 
 consists of the following (overlapping) classes of quivers:
 rank $2$ quivers, quivers originating from  surfaces (see Section~\ref{background}),
 quivers of finite type (i.e., with an orientation of a finite type Dynkin diagram in the mutation class), 
quivers of affine type (ones with an orientation of an affine type Dynkin diagram in the mutation class),
quivers of extended affine types $E^{(1,1)}_6$, $E^{(1,1)}_7 $ and $E^{(1,1)}_8$ (see Fig.~\ref{mut_fin}),
exceptional quivers of types $X_6$ and $X_7$ (see also Fig.~\ref{mut_fin}).

Another easy observation is that it is enough to consider just one frozen vertex. Indeed, as there are no arrows between frozen vertices, the frozen vertices do not affect each other in the process of mutations.

\begin{definition}
\label{def_b}  
Let $Q$ be a quiver of finite mutation type (with vertices $v_1,\dots,v_n$ all being mutable).  
Let $q$ be an additional  (frozen) vertex, denote by $b_i$ the number of arrows connecting a vertex $v_i$ of $Q$ to $q$.
We will say that the integer coefficient vector $\b=(b_1,\dots,b_n)$ is
  {\it admissible}  if $\b\ne 0$ and the quiver spanned by $Q$ and $q$ with the unique frozen vertex $q$ is of finite mutation type.   

\end{definition}

Therefore, the question of classification of mutation-finite exchange matrices is equivalent to finding all admissible vectors for every mutation-finite quiver without frozen vertices. 

\medskip

A distinguished class of cluster algebras consists of algebras of {\em finite type}: these were classified by Fomin and Zelevinsky in~\cite{FZ2} by establishing a connection with Cartan-Killing classification of simple Lie algebras. They also proved in~\cite{FZ4} that adding any coefficients to a cluster algebra of finite type results in a mutation-finite cluster algebra. Moreover, this characterizes cluster algebras of finite type: if every exchange matrix with given principal part is mutation-finite, then the principal part defines an algebra of finite type. A stronger conjecture was made in~\cite{FZ4} stating that it is sufficient to check the mutation-finiteness of the algebra with principal coefficients only, this was proved by Seven~\cite{S}.   

In particular, this provides the answer for the finite type.

\begin{prop}[\cite{FZ4}]
 If $Q$ is of finite type then  any vector $\b$ is  admissible.

\end{prop}

The next large class of quivers consists of quivers from surfaces~\cite{FST}. We first prove the following statement.

\begin{prop}[Theorem~\ref{t_per}]
  \label{p_t_per}
 If $Q$ is arising from a surface then  $\b$ is admissible if and only if it corresponds to a peripheral lamination. 

\end{prop}

Due to results of Gu~\cite{Gu1}, Proposition~\ref{p_t_per} provides an algorithm which determines whether a given quiver from a surface with a frozen vertex is mutation-finite: using~\cite{Gu1}, one can reconstruct a triangulation, then one can reconstruct a lamination using a procedure from~\cite{FT}, and then it is straightforward to check whether a given lamination is peripheral. We will give a more explicit characterization of coefficient vectors corresponding to peripheral laminations in Section~\ref{q-surfaces}.  

Next, we consider affine and exceptional mutation-finite classes. Every mutation class of quivers of affine type contains a representative with a double arrow, so the main tool in the considerations is the following necessary condition (which we call  the {\it annulus property}).

\begin{prop}[Corollary~\ref{cor-ann}]
  \label{prop-ann}
Let $Q$ be a quiver containing a double arrow from $v_1$ to $v_2$. Then a vector $\b$ is admissible only if $b_1=-b_2\le 0$.
\end{prop}

In the affine case $\t A$ we use Proposition~\ref{p_t_per} to show that the annulus property is also sufficient (see Lemma~\ref{l_ann} and Remark~\ref{a-any}).  

The same result applies to other affine quivers, but here their treatment is based on their cluster modular groups studied in~\cite{KG}. 

\begin{prop}[Theorem~\ref{t_af_e}]
For the representatives of the mutation classes of affine types $\t D$ and $\widetilde E$ shown in Fig.~\ref{af_e678}, a vector is admissible if and only if it satisfies the annulus property.

\end{prop}

This result is then generalized to all quivers of affine type containing a double arrow.

\begin{prop}[Theorem~\ref{aff_gen}]
If $Q$ is a quiver of affine type containing a double arrow, then a vector $\bm b$ is
admissible if and only if $\bm b$ satisfies the annulus property.
\end{prop}

For the extended affine quivers and quivers of type $X_7$ we take a specific representative $Q$ from the mutation class (see Figs.~\ref{e6_11}--\ref{x7})  and an element of the cluster modular group $\varphi$ to show that the annulus property for $Q$ is not compatible with the annulus property for $\varphi(Q)$, which results in the following statement.

\begin{prop}[Theorems~\ref{t_e6_11}--\ref{t_x7}]
Let $Q$ be of type $E_6^{1,1}$, $E_7^{1,1}$, $E_8^{1,1}$ or $X_7$. Then there is no admissible vector $\b$.
\end{prop}  

The case of quiver $Q$ of type $X_6$ is different: it admits an admissible vector such that the quiver spanned by $Q$ and the frozen vertex $q$ is isomorphic to $X_7$. In fact, the quiver admits a series of admissible vectors as follows.

\begin{prop}[Theorem~\ref{t_x6}]
  Let $Q$ be a quiver of type $X_6$ with a single arrow from $v_0$ to $v_5$, where $v_5$ is a leaf (see Fig.~\ref{x7}). Then a vector $\boldsymbol{b}$ is admissible if and only if $b_5=-2b_0\ge 0$ and all other $b_i$ vanish. Admissible vectors for all representatives of the mutation class are shown in Fig.~\ref{pic-x6}.
\end{prop}

In the next proposition we consider the quivers of rank $2$ (note that the first two parts have been already considered previously).  

\begin{prop}[Theorem~\ref{t_rk2}]
  Let $Q$ be a rank two quiver with the arrow from $v_1$ to $v_2$ of weight $a>0$. Let $\b=(b_1,b_2)$ be an integer vector. Then
  \begin{itemize}
  \item[(1)] if $a=1$ then $\b$ is admissible for any $b_1,b_2$; 
  \item[(2)] if $a=2$ then $\b$ is admissible if and only if $b_1=-b_2\le 0$;
  \item[(3)] if $a>2$ then there are no admissible vectors. 
    
  \end{itemize}

\end{prop}  

Finally, we specify a particular triangulation for every surface and give the admissibility criterion for the corresponding quiver, see Theorem~\ref{t_standard}. The criterion is also based on  the annulus property.

The results are extended to the general skew-symmetrizable case in Section~\ref{symmetrizable} by using diagrams in place of quivers and orbifolds in place of surfaces. The modified annulus property for arrows of weight $(1,4)$ is defined in Theorem~\ref{t_af_double4}.

\medskip
We now combine the results in one theorem.

\begin{theorem} Let $Q$ be a quiver/diagram and $\b=(b_1,\dots,b_n)$ be an integer vector.  
  \begin{itemize}
  \item[(1)] if $Q$ is of finite type then  any vector $\b$ is  admissible;
  \item[(2)] if $Q$ is of affine type and $Q$ contains a double arrow or an arrow of weight $(1,4)$, then a vector $\b$ is admissible if and only if $\b$  satisfies the annulus property; 
  \item[(3)] if $Q$ is arising from a surface/orbifold then  $\b$ is admissible if and only if it corresponds to some peripheral lamination; the criterion for admissibility is given for a specific representative of the mutation class  in Theorems~\ref{t_standard} and~\ref{t_standard_orb};
  \item[(4)] if $Q$ is of type $X_6$,  all possible admissible vectors are shown in Fig.~\ref{pic-x6};
  \item[(5)] otherwise, there is no admissible vector.  

  \end{itemize}  

\end{theorem}  

A criterion for being mutation-finite can be also reformulated in terms of the annulus property applied to the whole mutation class, this was proposed by Sergey Fomin. The statement is similar to the analogous criterion for quivers/diagrams without frozen vertices (see e.g.~\cite[Corollary 8]{DO}).

\begin{theorem}[Theorem~\ref{cr3}]
  Let $Q$ be a quiver/diagram with a frozen vertex $v$. Suppose that the subquiver $Q\setminus v$ is mutation-finite. 
  Then $Q$ is mutation-finite if and only if the annulus property holds in every quiver/diagram $Q'$ mutation-equivalent to $Q$ for every double arrow contained in $Q'\setminus v$.
 
\end{theorem}

From this one can conclude the following.

 \begin{cor}[Corollary~\ref{cor_cr}]
Let $Q$ be a quiver/diagram with a frozen vertex. Then
 $Q$ is mutation-finite if and only if for every quiver $Q'$ in the mutation class of $Q$ every rank $3$ subquiver/subdiagram of $Q'$ is mutation-finite.

 \end{cor}

\medskip
\noindent
The paper is organized as follows. In Section~\ref{background} we recall necessary background concerning triangulated surfaces and laminations on them.  Section~\ref{surfaces} is devoted to quivers from surfaces and the connection between admissible vectors and peripheral laminations. In Section~\ref{aff} we consider quivers of affine types, in Sections~\ref{extended} and~\ref{x} we treat extended affine  quivers and quivers of types $X_6$ and $X_7$. In the short Section~\ref{rk2} we consider quivers of rank 2. Section~\ref{q-surfaces} characterizes admissible vectors for a particular triangulation of a surface. In  Section~\ref{symmetrizable} all results are extended to the general context of skew-symmetrizable mutation classes. Finally, in Section~\ref{sec_cr} we discuss the criterion of mutation-finiteness in terms of the annulus property.

\subsection*{Acknowledgements}
We would like to thank Sergey Fomin for the question inspiring the current project and for helpful suggestions, Michael Shapiro for stimulating discussions, and Dani Kaufman for sharing then unpublished results of~\cite{KG} with us. We are grateful to the anonymous referee for useful remarks. A substantial part of the paper was written at the Isaac Newton Institute for Mathematical Sciences, Cambridge; we are grateful to the organizers of the program ``Cluster algebras and representation theory'', and to the Institute for support and hospitality during the program; this work was supported by EPSRC grant no EP/R014604/1.

\section{Background}
\label{background}

\subsection{Matrix mutation}

We start by reminding the definition of matrix mutation (we adopt the notation from~\cite{FZ4}).

Given an integer skew-symmetric $n\times n$ matrix $B=(b_{ij})$, the mutation $\mu_k$ of $B$ for $k\in \{ 1,\dots,n\}$ is defined by $\mu_k(B)=B'=(b_{ij}')$ where
$$
b_{ij}'=
\begin{cases}
  -b_{ij}& \text{if $i=k$ or $j=k$}\\
  b_{ij}+\sgn(b_{ik})[b_{ik}b_{kj}]_+& \text{otherwise,}\\
\end{cases}  
$$
where $\sgn(x)$ denotes the sign function and 
$[x]_+=\max\{x,0 \}.$

For an extended $m\times n$ matrix $\overline B$ with $m>n$ and skew-symmetric principal part given by first $n$ rows, the mutation is provided by the same formula.

\begin{remark}
A skew-symmetric $n\times n$ matrix $B=(b_{ij})$ can be represented by a quiver with $n$ vertices $v_1,\dots,v_n$ and $b_{ij}$ arrows from $v_i$ to $v_j$. Matrix mutation then can be reformulated in the quiver language, see Fig.~\ref{mut}.  

\end{remark}  

\begin{figure}[!h]
\begin{center}
  \epsfig{file=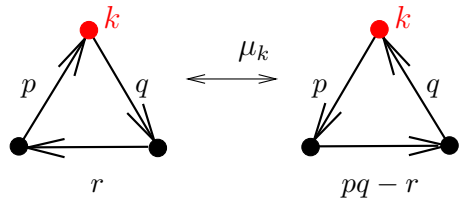,width=0.375\linewidth}
            \put(-166,26){\small $p$}
                \put(-56,26){\small $p$}
                 \put(-123,26){\small $q$}
                \put(-13,26){\small $q$}
                \put(-45,-12){\small $pq - r$}
                \put(-140,-12){\small $ r $}
                \put(-88,40){ $\mu_k$}
                \put(-135,50){\color{red} $k$}
                \put(-25,50){\color{red} $k$}     
\caption{Quiver mutation. Here $p,q$ are positive, and the sign of $r$ and $pq-r$ can be negative (which corresponds to opposite direction of the respective arrows).}
\label{mut}
\end{center}
\end{figure}

\subsection{Construction of quivers from triangulations}

We briefly recall the construction of quivers from triangulated surfaces~\cite{FST}.

Let $S$ be a connected orientable surface with boundary and with a finite set $M$ of marked points (such that every boundary component contains at least one marked point). Let $T$ be a triangulation of $S$ by the arcs having their endpoints in $M$.
Suppose that $T$ has no self-folded triangles (i.e. every triangle in $T$ is bounded by three distinct arcs or boundary segments). 
We construct a quiver $Q$ whose vertices $v_1,\dots, v_n$ correspond to the arcs $e_1,\dots,e_n$ of $T$.
The number of arrows in $Q$ from $v_i$ to $v_j$ is defined as

{\small
\begin{multline}b_{ij}=\#\{\text{triangles with sides $e_i$ and $e_j$, with $e_j$ following $e_i$ in clockwise order} \}-\\
  \#\{\text{triangles with sides $e_i$ and $e_j$, with $e_j$ following $e_i$ in counterclockwise order} \}.
\end{multline}
}

For more subtle rules for treating self-folded triangles see~\cite{FST}.

It is shown in~\cite{FST} that mutations of the quiver $Q$ correspond to flips of the triangulation $T$. It is easy to see from the definition above that combinatorially equivalent triangulations of $S$ give rise to isomorphic quivers (we say that triangulations are combinatorially equivalent if one can be taken to the other by an orientation-preserving homeomorphism of the surface). As it is shown in~\cite{Gu1}, a triangulation of a surface can be uniquely reconstructed from the corresponding quiver (up to finitely many low rank examples).

\subsection{Classification of mutation-finite quivers} 

We will heavily use the following.

\begin{theorem}[\cite{FeSTu1}]
  A connected mutation-finite quiver is either of rank 2, or a quiver arising from a triangulation of a surface, or
  a quiver mutation-equivalent to one of the eleven quivers $E_6$, $E_7$, $E_8$,  $\widetilde E_6$,  $\widetilde E_7$,  $\widetilde E_8$,  $E^{(1,1)}_6$, $E^{(1,1)}_7 $, $E^{(1,1)}_8$, $X_6$, $X_7$
 shown on Fig.~\ref{mut_fin}.
  
\end{theorem}

\begin{figure}[!h]
\begin{center}
  \epsfig{file=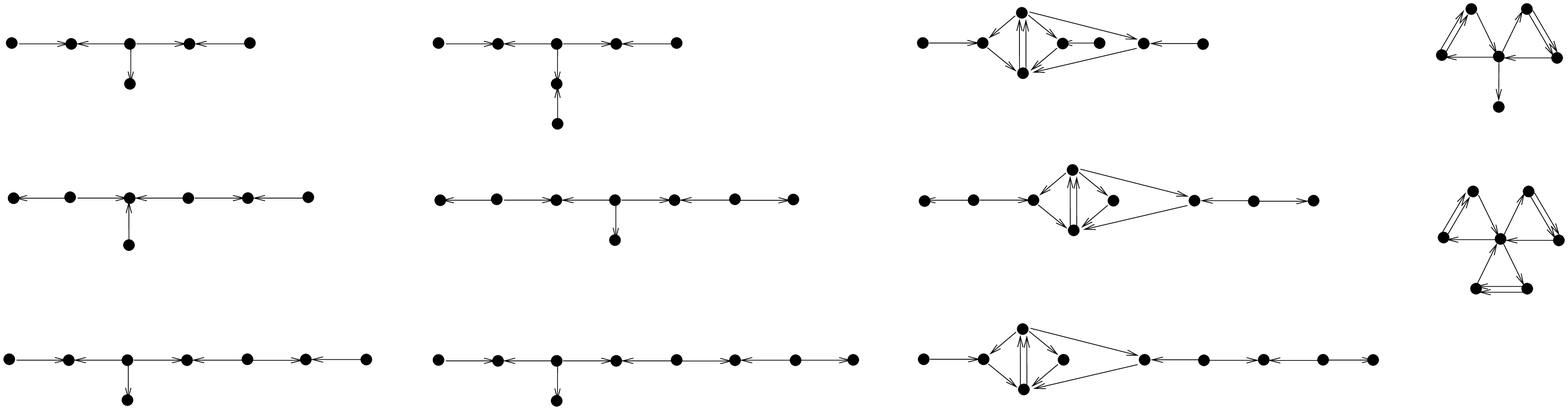,width=0.99\linewidth}
  \put(-453,82){ $E_6$}
  \put(-453,38){ $E_7$}
  \put(-453,-7){ $E_8$}
  \put(-330,82){ $\widetilde E_6$}
  \put(-330,38){ $\widetilde E_7$}
  \put(-330,-7){ $\widetilde E_8$}
  \put(-189,82){\small  $\widetilde E_6^{(1,1)}$}
  \put(-189,38){\small  $\widetilde E_7^{(1,1)}$}
  \put(-189,-7){\small  $\widetilde E_8^{(1,1)}$}
  \put(-45,77){$X_6$}
  \put(-45,27){$X_7$}
\caption{Eleven exceptional finite mutation classes}
\label{mut_fin}
\end{center}
\end{figure}

\subsection{Laminations as coefficients for surface case}
\label{lam}

It is shown in~\cite{FT} that in the case of a quiver from triangulated surface $S$, the coefficient vectors can be represented by {\it laminations} on the same surface $S$, and that the coefficient vectors can be computed from the triangulation and lamination using {\it shear coordinates}.

\begin{definition}[\cite{FT}, Def.~12.1]
An {\it integral unbounded measured lamination}, or just a {\it lamination} for short, on a marked surface $(S,M)$ is a finite collection of non-self-intersecting and pairwise non-intersecting curves in $S$, modulo isotopy
relative to $M$, subject to the restrictions specified below. Each curve must be one
of the following:
\begin{itemize}
\item[-]  a closed curve (an embedded circle);
\item[-] a curve connecting two unmarked points on the boundary of $S$;
\item[-] a curve starting at an unmarked point on the boundary and, at its other
end, spiralling into a puncture (either clockwise or counterclockwise); 
\item[-]  a curve both of whose ends spiral into punctures (not necessarily distinct);
\end{itemize}
where the following types of curves are not allowed:
\begin{itemize}
\item[-] a curve that bounds an unpunctured or once-punctured disk;
\item[-] a curve with two endpoints on the boundary of $S$ which is isotopic to a
piece of boundary containing no marked points, or a single marked point;
\item[-] a curve with two ends spiralling into the same puncture in the same direction without enclosing anything else.
\end{itemize}

\end{definition}

When speaking about two curves $\gamma_1$ and $\gamma_2$ (for example an arc of triangulation and a curve from a lamination) we always assume that the number of crossings is minimal possible for the curves in the homotopy classes of 
 $\gamma_1$ and $\gamma_2$ respectively. 

\begin{definition}[\cite{FT}, Def.~12.2]
  Let $L$ be a lamination and let $T$ be a triangulation without self-folded triangles of the same surface. For each
  arc $\gamma$ in $T$, the corresponding {\it shear coordinate} of $L$ with respect to the triangulation $T$, denoted by
  $b_\gamma (T,L)$, is defined as a sum of contributions from all intersections of curves in $L$ with the arc $\gamma$. Specifically, such an intersection contributes $+1$
(resp., $-1$) to $b_\gamma (T,L)$ if the corresponding segment of a curve in $L$ cuts through
the quadrilateral surrounding $\gamma$ cutting through edges  as shown in Fig.~\ref{shear} on the left (resp., on the right). Note
that at most one of these two types of intersection can occur. Note also that even
though a spiralling curve can intersect an arc infinitely many times, the number of
intersections that contribute to the computation of $b_\gamma (T,L)$ is always finite.

Shear coordinates can also be defined for arcs involved in self-folded triangles, see~\cite[Section~13]{FT}.
\end{definition}

Note that   the vector $\b=(b_1,\dots,b_n)$ from Definition~\ref{def_b} consists of negative shear coordinates of the corresponding lamination.

\begin{figure}[!h]
\begin{center}
  \epsfig{file=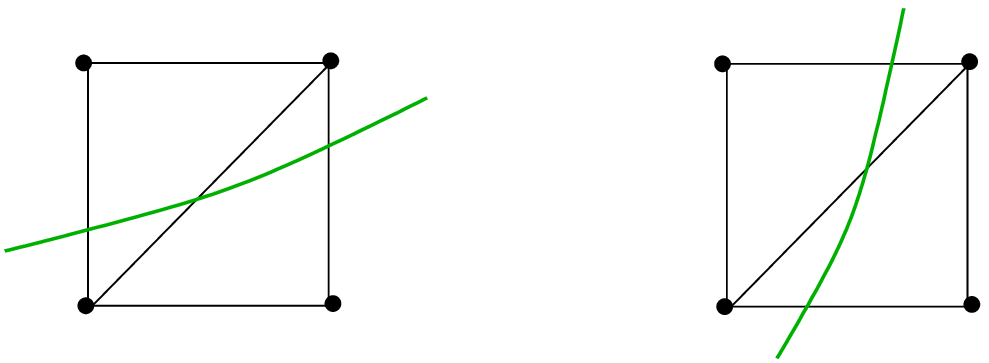,width=0.5\linewidth}
  \put(-233,33){ $+1$}
  \put(-23,83){ $-1$}
  \put(-53,27){\small $\gamma$}
  \put(-173,57){\small $\gamma$}
\caption{Shear coordinates.  }
\label{shear}
\end{center}
\end{figure}

It is known (see~\cite[Theorem~13.6]{FT}, see also~\cite[Section~3]{FG}) that 
 for a given triangulation $T$, the map
$L\mapsto (b_\gamma (T,L))_{\gamma\in T}$ 
provides a bijection between laminations and $\Z^n$.

In particular, for every triangulation $T$ and every arc $\gamma_0\in T$   there exists an {\it elementary lamination} $L$ such that  $b_{\gamma_0}(T,L)=1$ and $b_{\gamma_i}(T,L)=0$ for all $\gamma_i\in T$, $i\ne 0$. This elementary lamination consists of one curve which follows $\gamma_0$ but has its endpoints changed: for endpoints of $\gamma_0$ at a boundary marked point, the end of the elementary lamination is shifted to the left along the boundary, for endpoints of $\gamma_0$ at a puncture, the end of the elementary lamination is spiralling into the puncture anti-clockwise if the end is untagged and clockwise otherwise. We will also use {\it negative elementary lamination} defined by   $b_{\gamma_0}(T,L)=-1$ and $b_{\gamma_i}(T,L)=0$ for all $\gamma_i\in T$, $i\ne 0$. The negative elementary lamination also consists of one curve tracing the arc $\gamma_0$, but having boundary endpoints shifted to the right and the puncture end points spiralling to the puncture in the clockwise direction.

  \section{Quivers from surfaces and peripheral laminations}
  \label{surfaces}

  Let $Q$ be a quiver constructed by a triangulation $T$ of a surface $S$. As it was mentioned in Section~\ref{lam}, choosing a coefficient vector $\b$ is equivalent to a choice of a lamination $L$ on $S$. Since mutations correspond to flips of triangulations (and we can reach every triangulation by a sequence of flips), the vector $\b$ is admissible if and only if the shear coordinates of the lamination $L$ on all triangulations of $S$ take finitely many values only.

  \begin{definition}
    \label{per}
A curve on a marked surface $S$ will be called {\it peripheral} if it belongs to some lamination on $S$ and can be isotopically deformed to (a part of) a boundary component of $S$.
By a {\it peripheral lamination} we understand a lamination consisting of peripheral curves.
\end{definition}

In this section, we show that admissible vectors are in bijection with peripheral laminations (see Theorem~\ref{t_per}). In Section~\ref{q-surfaces} we will reformulate the result in terms of quivers.

\begin{theorem}
  \label{t_per}
  Let $Q$ be a quiver from a triangulated surface $S$. Then
admissible vectors for $Q$ are in bijection with  peripheral laminations on $S$.

\end{theorem}

\begin{proof}
  We need to show that the vector of shear coordinates of a lamination $L$ takes finitely many values if and only if $L$ is peripheral.
  
First, consider a peripheral lamination $L$. It is preserved by any Dehn twist along any closed curve  on the surface, and hence, it is preserved by the whole mapping class group of the surface (as the latter is generated by twists).  

  Observe that for a given surface $S$ there is only a finite number of combinatorial types of triangulations (in particular, this is precisely the reason why quivers originating from surfaces are mutation-finite), and combinatorially equivalent triangulations can be taken to each other by elements of the mapping class group of $S$. This implies that given an initial triangulation $T$ and the corresponding quiver $Q$, there is a finite number of mutation sequences applying which together with elements of the mapping class group we can reach any triangulation of $S$. Since shear coordinates of $L$ are invariant under the action of the mapping class group, this implies that the vector of shear coordinates of $L$ takes one of finitely many values.

Now, consider a lamination $L$ which is not peripheral. Then there exists a closed curve $C$ crossing $L$. Let $T$ be a triangulation and $D_C^{k}(T)$, $k\in \Z$  be the images of $T$ under iterative applications of  Dehn twist $D_C$ along $C$. We claim that shear coordinates of $L$ with respect to  $D_C^{k}(T)$ take infinitely many different values.
Indeed, to apply  $D_C$ to $T$ with keeping $L$ intact is the same as applying  $D_C^{-1}$ to $L$  and preserving $T$.
As $C$ intersects $L$, the Dehn twists  $D_C^{-k}(L)$ will produce infinitely many different laminations.
Due to the bijection between laminations and their shear coordinates, this implies that the shear coordinates of laminations $D_C^{-k}(L)$ with respect to triangulation $T$ are different. Hence, the shear coordinates of $L$ with respect to  triangulations   $D_C^{k}(T)$ are different, and thus take infinitely many values. 
This implies that non-peripheral laminations do not correspond to admissible vectors.
\end{proof}  

Notice that if a surface has no boundary, then it contains no peripheral curves. This gives rise to the following corollary of Theorem~\ref{t_per}.

\begin{cor}
  \label{no_bdry}
If a surface has no boundary, then the quiver of any of its triangulations has no admissible vectors. 

  \end{cor}

\begin{example}
\label{annulus}
Let $Q$ be the affine quiver $\widetilde A_1$ (two vertices $v_1$ and $v_2$ connected by a double arrow from $v_1$ to $v_2$). It corresponds to an annulus with one marked point on each boundary component, see  Fig.~\ref{ann}. The only peripheral curve on the annulus coincides with the unique closed curve inside this annulus (here we use the fact that every boundary component contains only one marked point). So, every peripheral lamination consists of an integer number of copies of this closed curve.
As one can see from Fig.~\ref{ann}, the corresponding coefficient vector satisfies $$b_1=-b_2\le 0$$
(recall that $b_i$ denotes negative shear coordinate, i.e. the number of arrows from a vertex $v_i$ to the frozen vertex).
  
\end{example}

The result of Example~\ref{annulus} can be reformulated as follows.

\begin{cor}[Annulus property]
\label{cor-ann}  
  Let $Q$ be the rank 2 quiver with a double arrow from $v_1$ to $v_2$ and $\b=(b_1,b_2)$ be an admissible vector. Then $b_1=-b_2\le 0$.
This will be called the {\it annulus property for $v_1\!=\!\!>\!v_2$}.
\end{cor}  

Corollary~\ref{cor-ann} leads to the following necessary condition for a coefficient vector $\b$ to be admissible, which we will heavily use throughout the paper.

\begin{definition}[Annulus property]
For an arbitrary quiver $Q$, a coefficient vector $\b=(b_1,\dots,b_n)$ satisfies  the {\it annulus property}  if for  every double arrow  $v_i\!=\!\!>\!v_j$ in $Q$ we have  $b_i=-b_j\le 0$.
\end{definition}

\begin{figure}[!h]
\begin{center}
  \epsfig{file=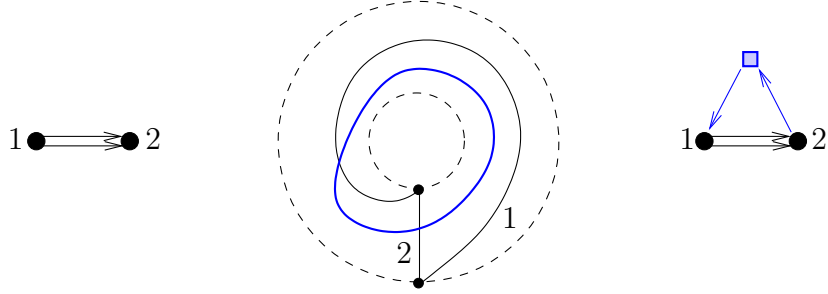,width=0.65\linewidth}
  \put(-303,53){$1$}
  \put(-251,53){$2$}
  \put(-116,23){$1$}
  \put(-156,10){$2$}
  \put(-49,53){$1$}
  \put(1,53){$2$}
\caption{Quiver $\widetilde A_1$, annulus and admissible coefficients. }
\label{ann}
\end{center}
\end{figure}

\section{Quivers of affine type}
\label{aff}

Let $L$ be a lamination on an annulus. A curve $C\in L$ 
 is called {\it bridging} if it has endpoints on both boundary components of the annulus (in other words, if and only if it is not peripheral).

\begin{lemma}
\label{l_ann}  
  Let $S_{p,q}$ be the annulus with  $p$ and $q$ boundary marked points triangulated as in Fig.~\ref{ann_big}. Then a vector $(b_1,\dots,b_n)$ is admissible if and only if it satisfies the annulus property.

\end{lemma}

\begin{figure}[!h]
\begin{center}
  \epsfig{file=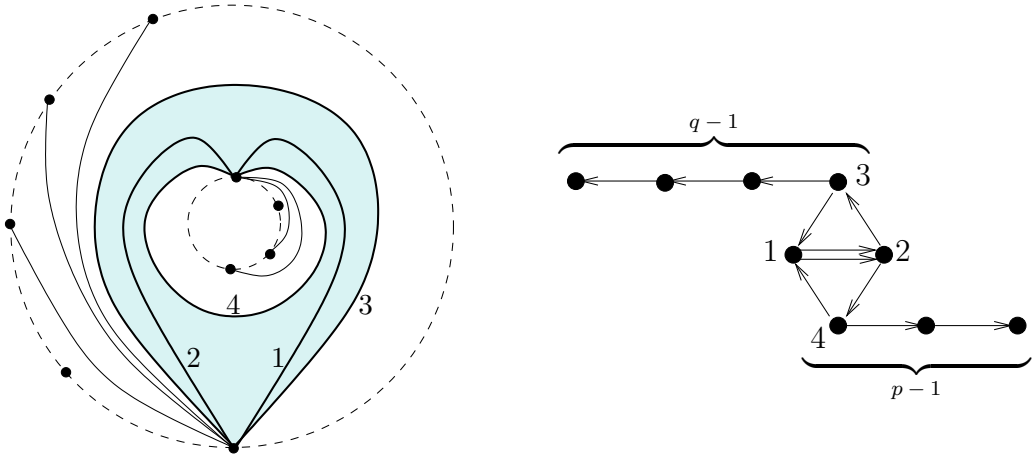,width=0.85\linewidth}
  \put(-303,53){\small $4$}
  \put(-253,53){\small $3$}
  \put(-318,33){\small $2$}
  \put(-286,33){\small $1$}
  \put(-65,102){$3$}
  \put(-82,40){$4$}
  \put(-100,73){$1$}
  \put(-50,73){$2$}
  \put(-89,37){  $ \underbrace{\phantom{aaaaaaaaaaaaaa}}_{\text{$p-1$}}   $    }  
  \put(-181,108){  $ \overbrace{\phantom{aaaaaaaaaaaaaaaaaaa}}^{\text{$q-1$}}   $    }  
\caption{Triangulated annulus $S_{p,q}$ with the corresponding quiver of type $\widetilde A_{p,q}$. }
\label{ann_big}
\end{center}
\end{figure}

\begin{proof}
  The annulus property $b_1=-b_2\le 0$ is necessary by Corollary~\ref{cor-ann}. We need to prove that it is also sufficient for admissibility of  $(b_1,\dots,b_n)$. In view of Theorem~\ref{t_per}, this is equivalent  to proving that for every vector   $(b_1,\dots,b_n)$ satisfying $b_1=-b_2\le 0$ there exists a peripheral lamination resulting in this vector. Since every vector is realisable by some lamination, we see that  it is sufficient to show
that a  lamination satisfying the condition  $b_1=-b_2\le 0$ cannot contain bridging curves.

Suppose that $L$ is a lamination on  $S_{p,q}$ with $b_1=-b_2\le 0$ and containing a bridging curve $l$.
Consider the restriction $\overline L$ of the lamination $L$ to the shaded annulus  $S_{1,1}$ with one marked point at each boundary component (see  Fig.~\ref{ann_big}).
The restriction $\bar l$ of the curve $ l$ to  $S_{1,1}$  is a bridging curve for  $S_{1,1}$. In Fig.~\ref{ann-bridge} we show a triangulated annulus (left) and its universal cover (right). For every bridging curve, we draw its lift (we normalize it by drawing the ``lower'' end in the same square of the universal cover)
and  compute its (negative) shear coordinates. 
 Notice that  every peripheral curve satisfies either $b_1=b_2=0$ (if it is not the closed curve) of $b_1=-b_2=-1$ otherwise. The latter is not contained in $L$ in presence of a bridging curve. Hence, peripheral curves in $L$ do not affect $b_1$ and $ b_2$, and it is sufficient
 to check coordinates of all collections of mutually non-intersecting bridging curves. 

 Any bridging curve on  $S_{1,1}$ can be obtained from any other bridging curve by application of a power of the Dehn twist along the unique closed curve, and if two curves differ by more than one twist then they intersect each other.  One can easily see that no bridging curve on  $S_{1,1}$ satisfies $|b_1|=|b_2|$, and coordinates $(b_1,b_2)$ of a pair of bridging curves  differing by one twist can take values $(-2k-1,2k+3), (1,1),(1,-1),(-1,-1),(-2k-2,2k+1)$ for $k\ge 0$ (see Fig.~\ref{ann-bridge}). Noone of these satisfies  $b_1=-b_2\le 0$, so we obtain a contradiction. 
\end{proof}

\begin{figure}[!h]
\begin{center}
  \epsfig{file=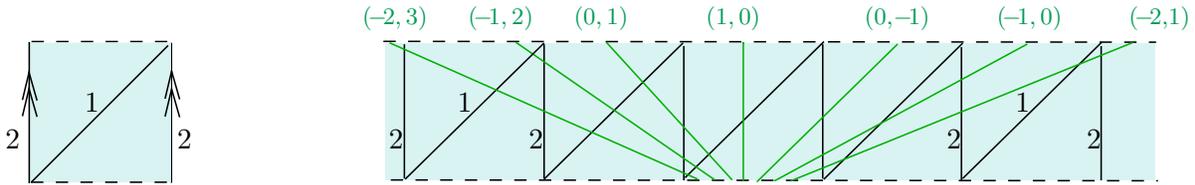,width=0.95\linewidth}
  \put(-438,14){\small $2$}
  \put(-373,14){\small $2$}
  \put(-408,28){\small $1$}
  \put(-267,28){\small $1$}
  \put(-293,14){\small $2$}
  \put(-240,14){\small $2$}
  \put(-56,28){\small $1$}
  \put(-82,14){\small $2$}
  \put(-29,14){\small $2$}  
  \put(-303,61){\color{ForestGreen} \scriptsize $(\!-\!2,3)$}
  \put(-263,61){\color{ForestGreen} \scriptsize $(\!-\!1,2)$}
  \put(-223,61){\color{ForestGreen} \scriptsize $(0,1)$}
  \put(-173,61){\color{ForestGreen} \scriptsize $(1,0)$}
  \put(-113,61){\color{ForestGreen} \scriptsize $(0,\!-\!1)$}
 \put(-63,61){\color{ForestGreen} \scriptsize $(\!-\!1,0)$}
\put(-13,61){\color{ForestGreen} \scriptsize $(\!-\!2,\!1)$}
  \caption{Triangulated annulus $S_{1,1}$, its universal cover,  and bridging curves with corresponding values of $(b_1,b_2)$ (recall that we define $b_i$ as negative shear coordinates).  }
\label{ann-bridge}
\end{center}
\end{figure}

\begin{remark}
  \label{a-any}
  Notice that the proof of Lemma~\ref{l_ann} does not use any properties of the triangulation of $S_{p,q}$ outside of the shaded annulus. In other words, the same proof works for any triangulation  of $S_{p,q}$ with the associated quiver containing a double arrow.   

  \end{remark}

We will now use Lemma~\ref{l_ann} to classify all admissible vectors for the remaining quivers of affine type.

Take the representatives of the mutation classes  $\t D_n$, $\widetilde E_6$, $\widetilde E_7$, $\widetilde E_8$ shown
in Fig.~\ref{af_e678}. A necessary condition on an admissible vector follows from the annulus property. 
The following theorem shows that every coefficient not breaking the annulus property is admissible.

\begin{figure}[!h]
\begin{center}
  \epsfig{file=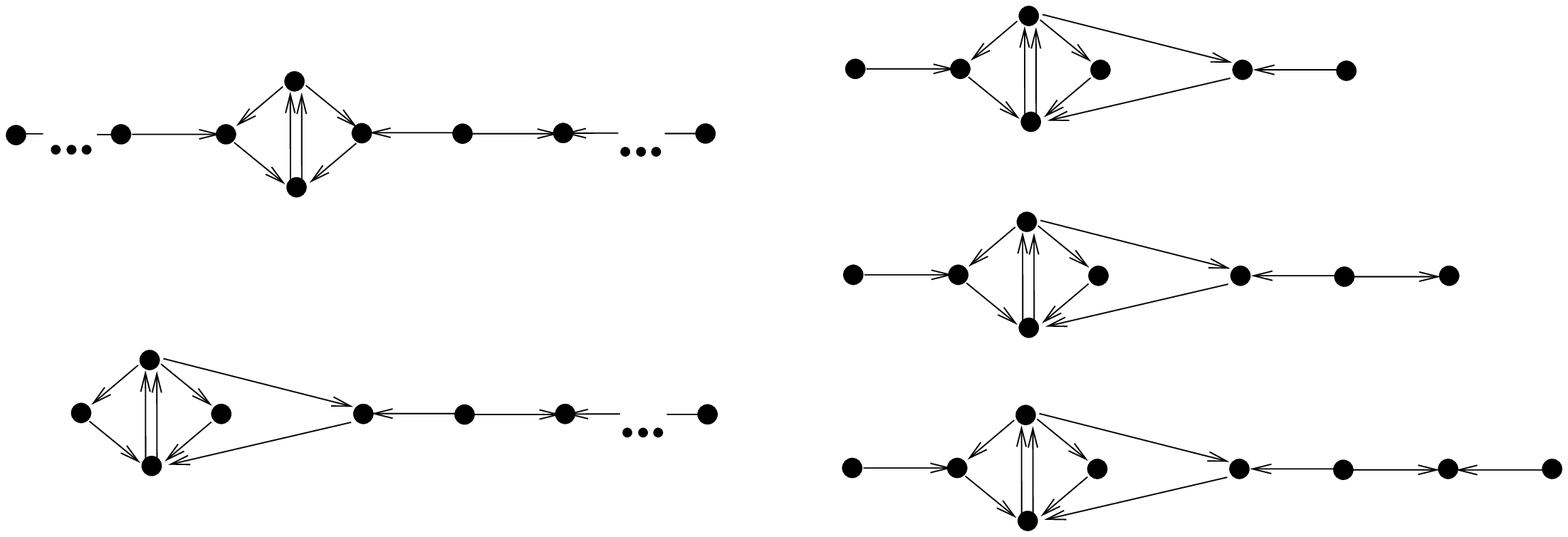,width=0.99\linewidth}
  \put(-215,160){$\widetilde E_6 $}
  \put(-162,151){\small $0$}
  \put(-162,109){\small $1$}  
  \put(-130,131){\small $2$}  
  \put(-181,138){\small $3$}  
  \put(-211,138){\small $4$}  
  \put(-97,138){\small $5$}  
  \put(-67,138){\small $6$}  
\put(-215,97){$\widetilde E_7 $}  
  \put(-162,92){\small $0$}
  \put(-162,50){\small $1$}  
  \put(-131,71){\small $2$}  
  \put(-180,78){\small $3$}  
  \put(-211,78){\small $4$}  
  \put(-97,78){\small $5$}  
  \put(-67,78){\small $6$}  
  \put(-38,78){\small $7$}  
\put(-215,40){$\widetilde E_8 $}  
  \put(-162,37){\small $0$}
  \put(-162,-8){\small $1$}  
  \put(-131,16){\small $2$}  
  \put(-180,23){\small $3$}  
  \put(-211,23){\small $4$}  
  \put(-97,23){\small $5$}  
  \put(-67,23){\small $6$}  
  \put(-38,23){\small $7$}  
  \put(-8,23){\small $8$}
  \put(-455,135){$\widetilde A_{p,q} $}
  \put(-353,100){  $ \underbrace{{\phantom{aaaaaaaaaaaaaaaaa}}}_{\text{$q-1$}}$    }  
\put(-455,100){  $ \underbrace{{\phantom{aaaaaaaaaaa}}}_{\text{$p-1$}}   $    }  
\put(-455,60){$\widetilde D_n $}  
  \put(-416,53){\small $0$}
  \put(-416,9){\small $1$}  
  \put(-384,32){\small $2$}  
  \put(-441,32){\small $n$}    
  \put(-348,39){\small $3$}    
  \put(-320,39){\small $4$}    
  \put(-290,39){\small $5$}    
  \put(-260,39){\small $n\!-\!1$}    
\caption{Representatives of affine mutation classes $\widetilde A_{p,q}$,  $\widetilde D_n$, $\widetilde E_{n}$. }
\label{af_e678}
\end{center}
\end{figure}

\begin{theorem}
 \label{t_af_e} 
Let $Q$ be the quiver of type  $\t D_n$, $\widetilde E_6$, $\widetilde E_7$ or $\widetilde E_8$ shown
in Fig.~\ref{af_e678}. A coefficient vector $\b$ is admissible if and only if it satisfies $b_0=-b_1\le 0$.

\end{theorem}  

Before proving the theorem, we recall the notion of the {\em cluster modular group} as the group generated by sequences of mutations (followed by permutations of the vertices of a quiver if needed) preserving the initial quiver (see e.g.~\cite{FeSTTu} for a detailed definition, where the term ``mapping class group of a cluster algebra'' is used instead, and~\cite{F,KG} for detailed descriptions of the cluster modular groups for affine and extended affine algebras).   

\begin{proof}

The necessity of the assumption of the theorem follows from the annulus property. We now prove the sufficiency.

  It is shown in~\cite{KG} that the cluster modular group for $Q$ is an abelian group generated by three mutation sequences (followed by certain permutations) described below. Define the sets of indices\\

  \noindent
  $I_\text{odd}=\begin{cases}
    \!\{i\in[5,k],i\,{\rm odd}\}& \text{\!\!for type $\t E_k$}\\
    \!\{i\in[3,n\!-\!1],i\,{\rm odd}\}& \text{\!\!for type $\t D_n$}
  \end{cases}
  $\ \quad\  
$    I_\text{even}=\begin{cases}
    \!\{i\in[5,k],i\,{\rm even}\}& \text{\!\!for type $\t E_k$}\\
    \!\{i\in[3,n\!-\!1],i\,{\rm even}\}& \text{\!\!for type $\t D_n$}
    \end{cases}
$\\

\noindent
and define the composite mutations $\mu_{\rm odd}$ and $\mu_{\rm even}$ as compositions of commuting mutations in $I_{\rm odd}$ and $I_{\rm even}$ respectively.

In these terms the generators of the cluster modular group can be written as follows:
\begin{itemize}
\item[] $\mu^{(1)}=\mu_2\circ\mu_1\circ\mu_0 $ with cyclic permutation $(v_2v_1v_0)$
\item[] $\mu^{(2)}=\begin{cases}
    \mu_4\circ\mu_3\circ\mu_1\circ\mu_0&  \text{with permutation } (v_3v_1v_0) \text{ for type } \t E_k\\
    \mu_{n}\circ\mu_1\circ\mu_0 & \text{with permutation } (v_{n}v_1v_0) \text{ for type } \t D_n
    \end{cases}$
\item[] $\mu^{(3)} = \mu_\text{even}\circ\mu_\text{odd}\circ\mu_1\circ\mu_0$ with permutation $(v_5v_1v_0)$ or $(v_3v_1v_0)$ for $\t E_k$ and $\t D_n$ resp. 
\end{itemize}

Inside $Q$ consider the following subquivers which we will call {\it wings} (we list the vertices of the subquivers in the brackets):
$$Q_1=\langle v_2\rangle,  \qquad Q_2=\langle v_3,v_4\rangle \text{ or } \langle v_{n}\rangle \text{ for $\t E_k$ and $\t D_n$ resp.}, \qquad Q_3=\langle v_{I_{\rm odd}},v_{I_{\rm even}} \rangle.
$$

We claim that each of $\mu^{(k)}$, $k=1,2,3$,  only changes the value of $b_i$ if $i\in Q_k$ and does not affect others.
To see this, consider the subquivers $Q\setminus Q_i=\langle v_0,v_1,Q_j,Q_k\rangle$, $i,j,k$ distinct. Each of these corresponds to a triangulated annulus, with an annulus $S_{1,1}$ inside (which corresponds to the subquiver $\langle v_0,v_1\rangle$) and  polygons  attached to each of its boundaries (which correspond to the wings).
By Lemma~\ref{l_ann}, the assumption $b_0=-b_1\le 0$ implies that the restriction of vector $\b$ on $Q\setminus Q_i$ is defined by some peripheral lamination on the corresponding annulus. The mutation  $\mu^{(k)}$ acts as a cyclic permutation of the boundary marked vertices of the triangulation corresponding to the $k$-th wing.  Therefore, this element acts trivially on the wing $Q_j$ and on the values of $b_0$ and $b_1$, while the order of the action on $Q_k$ is equal to the number of vertices in $Q_k$ plus one. Since we could choose the quiver $Q\setminus Q_j$ instead, $\mu^{(k)}$ acts trivially on $Q_i$ as well.
Thus, the action of the whole cluster modular group on the vector $\b$ has a finite orbit. 

The rest of the proof is similar to the proof of Theorem~\ref{t_per}. As $Q$ is mutation-finite, the number of distinct mutation sequences modulo the action of the cluster modular group is finite. Together with the finiteness of the orbit of $\b$ under the action of the cluster modular group this results in the admissibility of $\b$.
\end{proof}

We now generalize the result of Theorem~\ref{t_af_e} to all quivers of affine type containing a double arrow.

\begin{theorem}
\label{aff_gen}
If $Q$ is of affine type and $Q$ contains a double arrow, then a vector $\b$
is admissible if and only if $\b$ satisfies the annulus property.
\end{theorem}  

\begin{proof}
  For the quivers of types $\widetilde A$ the statement follows from Remark~\ref{a-any}. All quivers of type $\widetilde D$ are classified in~\cite{H}, and it follows from the classification that any quiver with a double arrow can be obtained from the quiver $Q$ of type $\t D_n$ shown in Fig.~\ref{af_e678} by mutations in vertices $v_i$ for $4\le i\le n$. Therefore, we can mutate our quiver to $Q$ preserving the annulus property, so Theorem~\ref{t_af_e} implies that the vector $\b$ is admissible. 

  The proof for types $\widetilde E_6$,  $\widetilde E_7$,  $\widetilde E_8$ is similar: the inspection of the mutation classes shows that all quivers with double arrows are obtained from the quivers in Fig.~\ref{af_e678} by a sequence of mutations at vertices $v_6,v_7,v_8$. As none of these vertices is connected to $v_0$ and $v_1$, such a sequence of mutations cannot break the annulus property.
  \end{proof}  

\section{Extended affine quivers}
\label{extended}

In this section, we prove that there are no admissible vectors for extended affine types $E_{6,7,8}^{(1,1)}$. For every mutation class we choose a specific representative containing a double arrow and find an element $\mu$ from the cluster modular group such that the application of $\mu$ breaks the annulus property. 

\subsection{Mutation class of $E_6^{(1,1)}$}
\begin{theorem}
\label{t_e6_11}  
There is no admissible vector for a quiver in the mutation class of $E_6^{(1,1)}$.

\end{theorem}  

\begin{proof}

It is sufficient to prove the statement for one quiver from the mutation class. We consider the quiver $Q$ shown in Fig.~\ref{e6_11}, left. Suppose that $\b=(b_1,\dots,b_8)$ is an admissible vector.

\medskip
\noindent
{\bf Plan of the proof and notation.}
The subquiver $\langle v_7,v_8\rangle$ is of type $\widetilde A_1$, so from the annulus property for $v_8=>v_7$  we have 
\begin{equation}
\label{(1)}  
  b_8=-b_7\le 0.
\end{equation}  

We will find a sequence of mutations $\mu$ taking $Q$ to the opposite quiver $Q^{op}$ (where $Q^{op}$ is obtained from $Q$ by reversing all arrows) and check that after the application of the mutation sequence $\mu$ to ${\bm b}$ the annulus property does not hold.

More precisely, let $$\mu_{*}=\mu_{3}\circ \mu_{2}\circ \mu_{1}  \quad \text{and} \quad 
\mu_\diamond=\mu_{7}\circ \mu_{6}\circ \mu_{5}\circ \mu_{4}$$
(notice that the components of each of these composite mutations commute) and consider the following sequence of three composite mutations:
$$ \mu=\mu_*\circ \mu_\diamond\circ \mu_*.
$$
Observe that  $Q_1:=\mu_*(Q)$ is the quiver shown in  Fig.~\ref{e6_11}, right, 
$Q_2:= \mu_\diamond\circ \mu_*(Q)=Q_1^{op}$ is the quiver opposite to $Q_1$, and  $Q_3:=\mu(Q)=Q^{op}$ is the quiver opposite to $Q$.

\begin{figure}[!h]
\begin{center}
  \epsfig{file=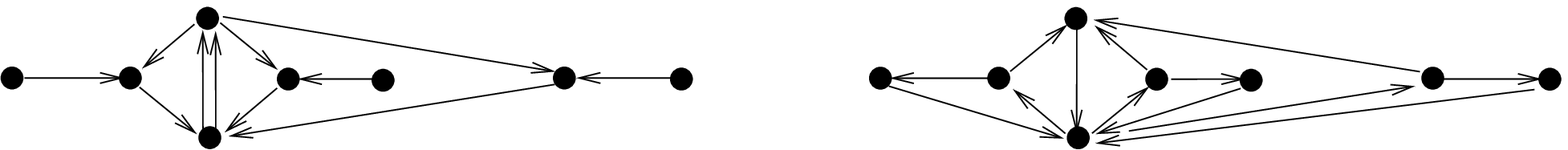,width=0.99\linewidth}
  \put(-455,25){\scriptsize 4}
  \put(-419,25){\scriptsize 1}
  \put(-371,25){\scriptsize 2}
  \put(-345,24){\scriptsize 5}
  \put(-293,25){\scriptsize 3}
  \put(-260,25){\scriptsize 6}
  \put(-399,40){\scriptsize 7}
  \put(-399,-4){\scriptsize 8}
  \put(-205,25){\scriptsize 4}
  \put(-169,25){\scriptsize 1}
  \put(-121,25){\scriptsize 2}
  \put(-95,24){\scriptsize 5}
  \put(-43,25){\scriptsize 3}
  \put(-10,25){\scriptsize 6}
  \put(-149,40){\scriptsize 7}
  \put(-148,-4){\scriptsize 8}
\caption{$Q=E_6^{1,1}$ (left) and $\mu_*(Q)$ (right). }
\label{e6_11}
\end{center}
\end{figure}

We now compute how the vector $\b$ changes under the sequence of mutations. Denote its components
by $b_i^{(1)}$,  $b_i^{(2)}$ and  $b_i^{(3)}$  after applying $\mu_*$,  $\mu_\diamond\circ \mu_*$ and   $\mu=\mu_*\circ \mu_\diamond\circ \mu_*$ respectively.

If $\b$ is an admissible vector, then  $(b_1^{(3)},\dots,b_8^{(3)})$ satisfies the annulus property for $v_7\!=\!\!>\!v_8$ in $Q_3$ (as $Q_3=Q^{op}$), i.e. one must have
\begin{equation}
\label{(2)}  
b_7^{(3)}=-b_8^{(3)}\le 0,
\end{equation}
but the computation will show this implies $\b= 0$.

\medskip
\noindent
{\bf Computation of  $b_8^{(3)}$.}
We start by computing $b_8^{(1)}$, $b_8^{(2)}$ and $b_8^{(3)}$:
\vspace{10pt}

\begin{itemize}\setlength\itemsep{1em}
\item[]   $b_8^{(1)}=b_8-[-b_1]_+ -[-b_2]_+ -[-b_3]_+ \le b_8 \le 0$;  
\item[]   $b_8^{(2)}=b_8^{(1)}-[-b_4^{(1)}]_+ -[-b_5^{(1)}]_+ -[-b_6^{(1)}]_+ -[-b_7^{(1)}]_+ \le b_8^{(1)} \le 0$;
\item[]   $b_8^{(3)}=b_8^{(2)}-[-b_1^{(2)}]_+ -[-b_2^{(2)}]_+ -[-b_3^{(2)}]_+ \le b_8^{(2)} \le 0$.

\end{itemize}

If $b_8^{(3)}\ne 0$ then we obtained  $b_8^{(3)}< 0$ which contradicts~(\ref{(2)}). Therefore  $b_8^{(3)}= 0$.
Furthermore, since $b_8\le 0$ and all summands above are also non-positive, the condition $b_8^{(3)}= 0$ is satisfied if and only if $b_8= 0$ and all entries in the computation above vanish. This results in the following constrains:
\begin{equation}
  \label{(3)}b_7=b_8=0,\quad b_1,b_2,b_3\ge 0,\quad b_4^{(1)},b_5^{(1)},b_6^{(1)},b_7^{(1)}\ge 0,\quad b_1^{(2)},b_2^{(2)},b_3^{(2)}\ge 0.
  \end{equation}

\medskip
\noindent
{\bf Computation of  $b_7^{(3)}$.}
Since  $b_8^{(3)}= 0$, (\ref{(2)}) implies that $b_7^{(3)}= 0$. Our goal is to express  $b_7^{(3)}$ via the components of $\b$ to find further constrains on $b_i$. We do this by first expressing $b_7^{(3)}$ via $b_i^{(2)}$, then computing required $b_i^{(2)}$ in terms of $b_j^{(1)}$, etc. While computing we will use the inequalities~(\ref{(3)}).

\begin{itemize}\setlength\itemsep{0.3em}
\item[]   $b_7^{(1)}=b_7+[b_1]_+ +[b_2]_+ +[b_3]_+ =b_1+b_2+b_3$;  
\item[]   $b_7^{(2)}=-b_7^{(1)}=-b_1-b_2-b_3$;
\item[]   $b_7^{(3)}=b_7^{(2)}+[b_1^{(2)}]_+ +[b_2^{(2)}]_+ +[b_3^{(2)}]_+ =b_7^{(2)}+b_1^{(2)} +b_2^{(2)}+b_3^{(2)}$.\\
\end{itemize}

\begin{itemize}\setlength\itemsep{0.3em}
\item[]  $b_i^{(1)}=-b_i\le 0$ for $i=1,2,3$;
\item[]  $b_4^{(1)}=b_4+b_1\ge 0$;  \qquad  $b_5^{(1)}=b_5+b_2\ge 0$; \qquad  $b_6^{(1)}=b_6+b_3\ge 0$; 
\item[]  $b_7^{(1)}=b_1+b_2+b_3\ge 0$;
\vspace{10pt}
\item[]  $b_1^{(2)}=b_1^{(1)}+b_4^{(1)}+b_7^{(1)}=(-b_1)+(b_4+b_1)+(b_1+b_2+b_3=b_4+b_1+b_2+b_3)\ge 0$;
\item[]  $b_2^{(2)}=b_2^{(1)}+b_5^{(1)}+b_7^{(1)}=b_5+b_1+b_2+b_3\ge 0$;
\item[]  $b_3^{(2)}=b_3^{(1)}+b_6^{(1)}+b_7^{(1)}=b_3+b_1+b_2+b_3\ge 0$.
\end{itemize}

Finally, we obtain
$$
b_7^{(3)}=b_7^{(2)}+b_1^{(2)} +b_2^{(2)}+b_3^{(2)}=(b_4+b_1)+(b_5+b_2)+(b_6+b_3)+b_1+b_2+b_3= 0.
$$

Notice that every summand in the sum is non-negative. Therefore, every summand is zero, in particular, $b_1=b_2=b_3=0$, from which we have
$b_4=b_5=b_6=0$. Since we also know $b_7=b_8=0$, we conclude that $\b=0$, which implies there are no non-zero admissible vectors.
\end{proof}

\subsection{Mutation class of $E_7^{(1,1)}$}
\begin{theorem}
There is no admissible vector for a quiver in the mutation class of $E_7^{(1,1)}$.

\end{theorem}  

\begin{proof}
  The proof follows the same scheme as the one for the case of $E_6^{(1,1)}$, we omit explicit computations as they are very similar to the previous case but much longer.
  
  We consider the quiver $Q$  shown in Fig.~\ref{e7_11}.
  Denote
  $$\mu_{*}=\mu_5\circ \mu_4 \circ \mu_3 \circ \mu_2 \circ \mu_1  \qquad \qquad \mu_\diamond=  \mu_8 \circ \mu_7 \circ \mu_6 $$
  and consider
  $$ \mu =\mu_*\circ \mu_\diamond \circ \mu_* \circ \mu_\diamond \circ \mu_*.$$

  Let $\b=(b_1,\dots,b_9)$ be an admissible vector. Denote by $\b'=(b_1',\dots,b_9')$ the result of mutation $\mu$. 
  One can check that $\mu(Q)=Q^{op}$.  Then by the annulus property for $Q$ we have
  $$ b_9=-b_8<=0,$$
  and by the annulus property for $\mu(Q)=Q^{op}$ we need
  $$ b_8'=-b_9'<=0.$$

  A computation similar to the one for $E_6^{(1,1)}$ shows that $b_9'\le b_9\le 0$, which implies $b_9'=0=b_8'$ and similar constrains on the summands. 
Computing then $b_8'$ in exactly the same way as for  $E_6^{(1,1)}$ we conclude that all $b_i=0$ for $i=1,\dots,9$.

\begin{figure}[!h]
\begin{center}
  \epsfig{file=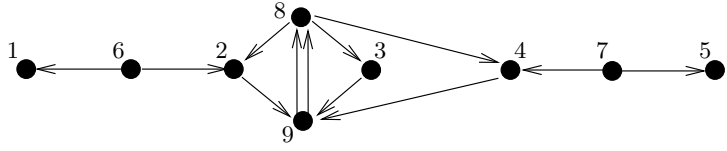,width=0.59\linewidth}
  \put(-272,28){\scriptsize 1}
  \put(-232,28){\scriptsize 6}
  \put(-193,28){\scriptsize 2}
  \put(-133,28){\scriptsize 3}
  \put(-80,28){\scriptsize 4}
  \put(-49,28){\scriptsize 7}
  \put(-10,28){\scriptsize 5}
  \put(-171,43){\scriptsize 8}
  \put(-168,-4){\scriptsize 9}
\caption{$Q=E_7^{1,1}$. }
\label{e7_11}
\end{center}
\end{figure}
\end{proof}

\subsection{Mutation class of $E_8^{(1,1)}$}
\begin{theorem}
  \label{t_e8_11}
  There is no admissible vector for a quiver in the mutation class of $E_8^{(1,1)}$.

\end{theorem}  

\begin{proof}
  The proof is very similar to the one for $E_6^{(1,1)}$ and $E_7^{(1,1)}$.
  We consider the quiver $Q$  shown in Fig.~\ref{e8_11}.
  Denote
  $$\mu_{*}=\mu_8\circ \mu_6 \circ \mu_4 \circ \mu_3 \circ \mu_2  \qquad \qquad \mu_\diamond=  \mu_9 \circ \mu_7\circ \mu_5 \circ \mu_1 $$
  and consider
  $$ \mu =\mu_*\circ \mu_\diamond \circ \mu_* \circ \mu_\diamond \circ \mu_*\circ \mu_\diamond \circ \mu_* \circ \mu_\diamond \circ \mu_*.$$
As before, $\mu(Q)=Q^{op}$, and an explicit computation shows that the annulus property does not hold for $\mu(Q)$ unless $\b=0$. 

\begin{figure}[!h]
\begin{center}
  \epsfig{file=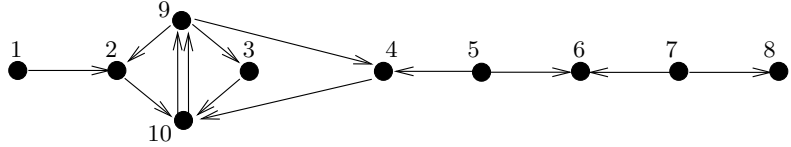,width=0.65\linewidth}
  \put(-295,28){\scriptsize 1}
  \put(-259,28){\scriptsize 2}
  \put(-207,28){\scriptsize 3}
  \put(-153,28){\scriptsize 4}
  \put(-122,28){\scriptsize 5}
  \put(-82,28){\scriptsize 6}
  \put(-47,28){\scriptsize 7}
  \put(-10,28){\scriptsize 8}
  \put(-239,43){\scriptsize 9}
  \put(-243,-4){\scriptsize 10}
\caption{$Q=E_8^{1,1}$. }
\label{e8_11}
\end{center}
\end{figure}
  \end{proof}  

\section{Mutation classes of $X_6$ and $X_7$}
\label{x}

In this section, we show that there are no admissible vectors for mutation classes $X_6$ and $X_7$. The proof is similar to the one for extended affine quivers.

\begin{theorem}
  \label{t_x7}
There is no admissible vector for a quiver in the mutation class of $X_7$.
\end{theorem}

\begin{proof}
  The idea is similar to the one we used for the case of $E_6^{(1,1)}$.
  
  Consider the quiver $Q$ of type $X_7$ shown in Fig.~\ref{x7}. From the annulus property for three double arrows we get
$$
b_1=-b_2\le 0 \qquad b_3=-b_4\le 0 \qquad b_5=-b_6\le 0.
$$

  The composition of mutations
  $$\mu_{012}=\mu_2\circ \mu_1 \circ \mu_0$$ takes $Q$ to an isomorphic quiver with different 
location of double arrows (after $\mu_{012}$ the double arrows will be $v_0v_2, v_3v_6$ and $v_4v_5$). 
The annulus property for the mutated quiver $\mu_{012}(Q)$ after computing all entries would result in the following equations:
$$ b_3+b_6+2b_0=0 \qquad \qquad  b_4+b_5+2b_0=0$$
(the third equation will be $b_1+b_2=0$ which is satisfied automatically).

By symmetry, an application of another composition of mutations $\mu_{034}$ leads to the equations
$$ b_1+b_6+2b_0=0 \qquad \qquad  b_2+b_5+2b_0=0,$$
and, similarly, one obtains  from $\mu_{056}$ that
$${\phantom{.}}b_3+b_2+2b_0=0 \qquad \qquad  b_4+b_1+2b_0=0.$$

Adding all six equations together we get
$$2(b_1+b_2)+2(b_3+b_4)+2(b_5+b_6)+12b_0=0
$$
which implies $b_0=0$, and thus the six equations above result in $b_1=b_3=b_5=-b_2=-b_4=-b_6\le 0$.

So far, we have only used equalities arising from the annulus property but not the inequalities. Computing the value of $b_3$ after mutation $\mu_{012}$ (call it $b_3'$), one can find that if all assumptions from above hold then $b_3'\le 0$, while from the annulus property for $\mu_{012}(Q)$ one gets $b_3'\ge 0$. This implies $b_3'=0$, which can hold in the only case of $b_3=0$ (similarly to $E$ case), and  hence $b_i=0$ for all $i\in\{0,1,2,\dots,6\}$.
\end{proof}

\begin{figure}[!h]
\begin{center}
  \epsfig{file=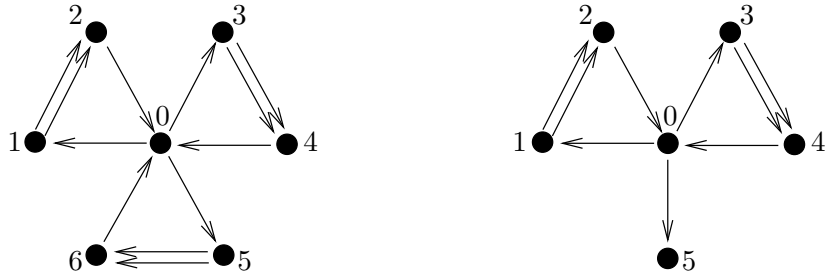,width=0.65\linewidth}
  \put(-246,55){\small 0}
  \put(-302,45){\small 1}
  \put(-279,93){\small 2}
  \put(-217,93){\small 3}
  \put(-190,45){\small 4}
  \put(-215,0){\small 5}
  \put(-279,0){\small 6}
  \put(-54,55){\small 0}
  \put(-111,45){\small 1}
  \put(-86,93){\small 2}
  \put(-25,93){\small 3}
  \put(2,45){\small 4}
  \put(-47,0){\small 5}
\caption{Quivers $X_7$ (left) and $X_6$ (right). }
\label{x7}
\end{center}
\end{figure}

\begin{theorem}
\label{t_x6}
  Let $Q$ be a quiver of type $X_6$ with a single arrow from $v_0$ to $v_5$, where $v_5$ is a leaf (see Fig.~\ref{x7}). Then a vector $\boldsymbol{b}$ is admissible if and only if $b_5=-2b_0\ge 0$ and all other $b_i$ vanish. Admissible vectors for all representatives of the mutation class are shown in Fig.~\ref{pic-x6}.
\end{theorem}

\begin{figure}[!h]
\begin{center}
 \includegraphics[width=.9\textwidth]{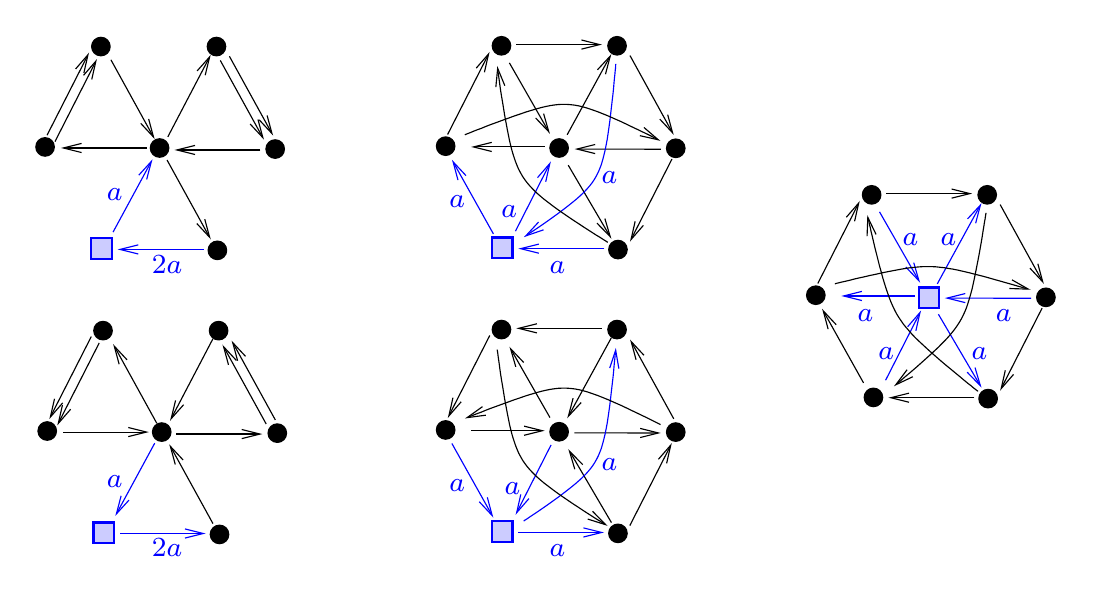}
\caption{Admissible vectors of all five quivers of type $X_6$, $a\in\mathbb N$.  }
\label{pic-x6}
\end{center}
\end{figure}

\begin{proof}
  Let $Q$ be as in Fig.~\ref{x7}, right. From the annulus property we have
  $$ b_1=-b_2\le 0  \qquad\text{and}   \qquad b_3=-b_4\le 0.$$

  \medskip
  \noindent
  {\bf Two equations from  $\mu_{012}$ and $\mu_{034}$.}
  Consider the sequences of mutations $\mu_{012}$ and $\mu_{034}$. By the annulus properties in the resulting quivers we get the following conditions:
  \begin{equation}\label{one}
    b_4+b_5+2b_0=0  \qquad \text{and}   \qquad b_2+b_5+2b_0=0. 
\end{equation}
    Notice that we do not need to compute anything here: the equations follow from the computation for $X_7$ restricted to $X_6$.
From this we conclude that $b_2=b_4$, i.e. $-b_1=-b_3=b_2=b_4=\beta$ for some $\beta\ge 0$.

  \medskip
  \noindent
  {\bf More equations from  $\mu_5$.}
To obtain more equations, we will first apply mutation $\mu_5$ to $X_6$, then $v_5$ will be a source instead of a sink, so the resulting quiver will be isomorphic to the subquiver of $X_7$ where the vertex $v_5$ is removed.

More precisely, denote by $b_i'$ the result of the application of $\mu_5$ to $\boldsymbol{b}$, and call the image of $b_5$ by $b_6'$ (to use the restriction of $X_7$). 
Then the mutations $\mu_{012}$ and $\mu_{034}$ will lead to the following two equations:
\begin{equation}
\label{two}
  b_3'+b_6'+2b_0'=0 \quad \text{and} \quad b_1'+b_6'+2b_0'=0.
\end{equation}
The entries here are computed from mutation $\mu_5$ as follows:
$$ b_1'=b_3'=-\beta  \qquad b_2'=b_4'=\beta \qquad  b_6'=-b_5  \qquad  b_0'=b_0+[b_5]_+,$$
and each of the equations in~(\ref{two}) leads to the following:
$$ -\beta-b_5+2b_0+2[b_5]_+=0.
$$
Since we also have $ \beta+b_5+2b_0=0$ from equation~(\ref{one}), we obtain the following equations:
$$
  4b_0+2[b_5]_+=0 \quad \text{and} \quad 2\beta+2b_5-2[b_5]_+=0,
$$
which can be simplified to $2b_0+[b_5]_+=0$ and $\beta=[-b_5]_+$. 

\medskip

  We now have two cases to consider: either $b_5\le 0$ or $b_5> 0$.
  If  $b_5\le 0$, then $b_0=0$ and $\beta=-b_5$, so we obtain a vector $$\boldsymbol{b}=(b_0,\dots,b_5)=(0,-\beta,\beta,-\beta,\beta,-\beta).$$
Applying mutation $\mu_{012}$, we obtain a quiver with a double arrow $v_4v_5$, and one can observe that the annulus property for this double arrow is not satisfied (unless $\beta=0$ which implies $\boldsymbol{b}=0$).

Therefore, we can assume $b_5>0$, so $\beta=0$ and $b_5=-2b_0$. We are left to show that all such vectors are admissible. The proof goes along the same lines as the proof of Theorem~\ref{t_af_e}: we check that every generator of the cluster modular group leaves vector $\boldsymbol{b}$ intact, where generators of the cluster modular group of the quiver of type $X_6$ shown in Fig.~\ref{x7} can be found in~\cite[Section 4.2]{I}.  

\end{proof}

\begin{remark}
  \label{x6-ann}
Observe that the violation of the annulus property is the only argument used in the proof of Theorem~\ref{t_x6} to show that a vector is not admissible. Therefore, we can conclude that for a quiver of type $X_6$ a vector is admissible if and only if the annulus property holds after every sequence of mutations. This observation will be generalized to all mutation-finite quivers in Section~\ref{sec_cr}. 
  \end{remark}

\section{Rank 2 quivers}
\label{rk2}

\begin{theorem}
  \label{t_rk2}
  Let $Q$ be a rank two quiver with the arrow from $v_1$ to $v_2$ of weight $a>0$. Let $\b=(b_1,b_2)$ be an integer vector. Then
  \begin{itemize}
  \item[(1)] if $a=1$ then $\b$ is admissible for any $b_1,b_2$; 
  \item[(2)] if $a=2$ then $\b$ is admissible if and only if $b_1=-b_2\le 0$;
  \item[(3)] if $a>2$ then there are no admissible vectors. 
    
  \end{itemize}
\end{theorem}

\begin{proof}
The first and second parts concern finite and affine types.

To prove the third part, notice that after at most two mutations (and swapping the labels of $v_1$ and $v_2$ if needed) we may assume that $Q=v_1\stackrel{a}\rightarrow v_2$, and $b_2\ge 0\ge b_1$. We may also assume that $|b_1|\ge |b_2|$ (otherwise replace $\mu_1$ with $\mu_2$ in the consideration below). 
Then after mutation $\mu_1$ we will get 
$$b_2'=b_2-a(-b_1)= b_2+ab_1<b_2+2b_1=(b_2+b_1)+b_1\le b_1,
$$
so, the absolute value of $b_2$ increases. Moreover, after swapping the labels of $v_1$ and $v_2$ the assumption above holds again, so we can mutate again to increase the components of the coefficient vector indefinitely.
\end{proof}

\section{Quivers from surfaces}
\label{q-surfaces}

In Section~\ref{surfaces} we gave a general characterization of admissible vectors via peripheral laminations. We now want to make this more explicit by describing admissible vectors for a special triangulation from every mutation class. We exclude from our consideration disks with at most two punctures and unpunctured annuli as these correspond to quivers of finite or affine type and thus were considered either in~\cite{FZ4} or in Section~\ref{aff}.      

If a surface has no boundary, then, by Corollary~\ref{no_bdry}, its quivers cannot have any admissible vector. Therefore, from now on we assume that a surface $S$ has at least one boundary component.

\begin{figure}[!h]
\begin{center}
  \epsfig{file=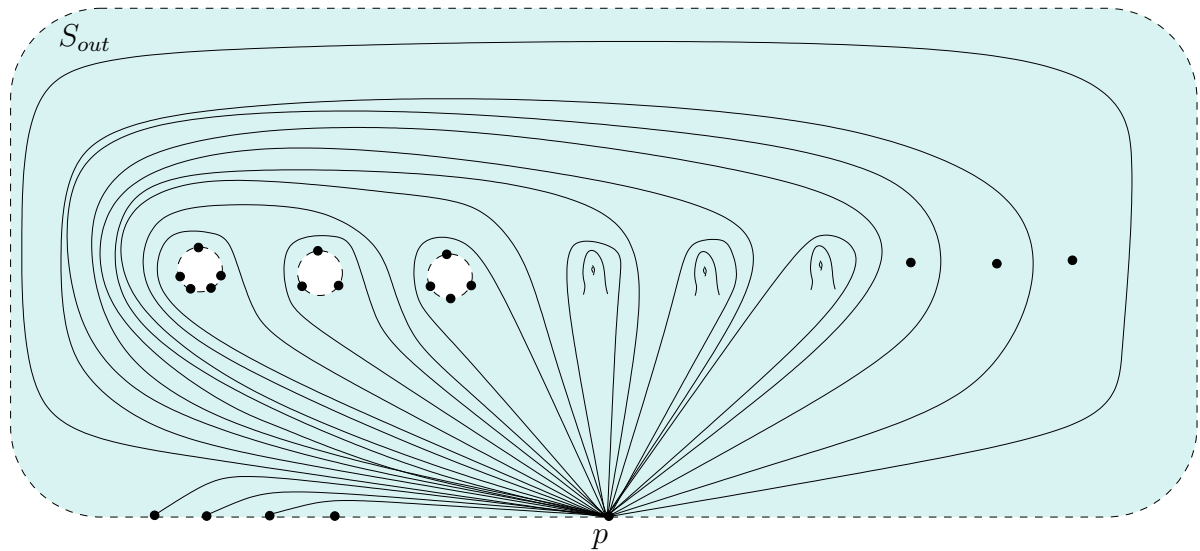,width=0.99\linewidth}
  \put(-432,180){$S_{out}$}
  \put(-230,-8){$p$}
\caption{Standard triangulation of a surface with at least one boundary component }
\label{st}
\end{center}
\end{figure}

A surface $S$ contains the following {\it features}: boundary components (each with a number of boundary marked points), punctures and handles. To construct the triangulation we do the following:

\begin{itemize}
\item[--] Choose any boundary component (we call it the {\it outer boundary component})  and a marked point $p$ on it.
All other boundary components will be called {\it inner} and the corresponding features will be called {\it holes}.  
\item[--] Place all  features along a line from left to right, first all holes, then all handles, then all punctures, as in  Fig.~\ref{st}, and enclose them by nested loops based at $p$ so that every feature (except for the leftmost one in Fig.~\ref{st}) lies inside a digon with both vertices at $p$ (recall that we excluded the case where S is a disc with one puncture).
\item[--] Triangulate the digons with features as follows:
  \begin{itemize}
  \item[-] each hole is enclosed by a loop $x_i$ and the domain inside $x_i$ triangulated as in Fig.~\ref{parts}, left;
  \item[-] each handle is enclosed by a loop $y_i$ and the domain inside $y_i$ triangulated as in Fig.~\ref{parts}, middle left;
  \item[-] each puncture inside a digon is connected by two arcs to two ends of the digon, see  Fig.~\ref{parts}, middle right;        
  \item[-] if there are no holes and handles, then the innermost monogon with two punctures is triangulated as in  Fig.~\ref{parts}, right;
  \item[-] if the outer boundary contains other marked points than $p$, then the outermost loop at $p$ separates a polygon (denote it $S_{out}$). $S_{out}$ is triangulated as shown in Fig.~\ref{st}. 
  \end{itemize}
\end{itemize}.

\begin{figure}[!h]
\begin{center}
  \epsfig{file=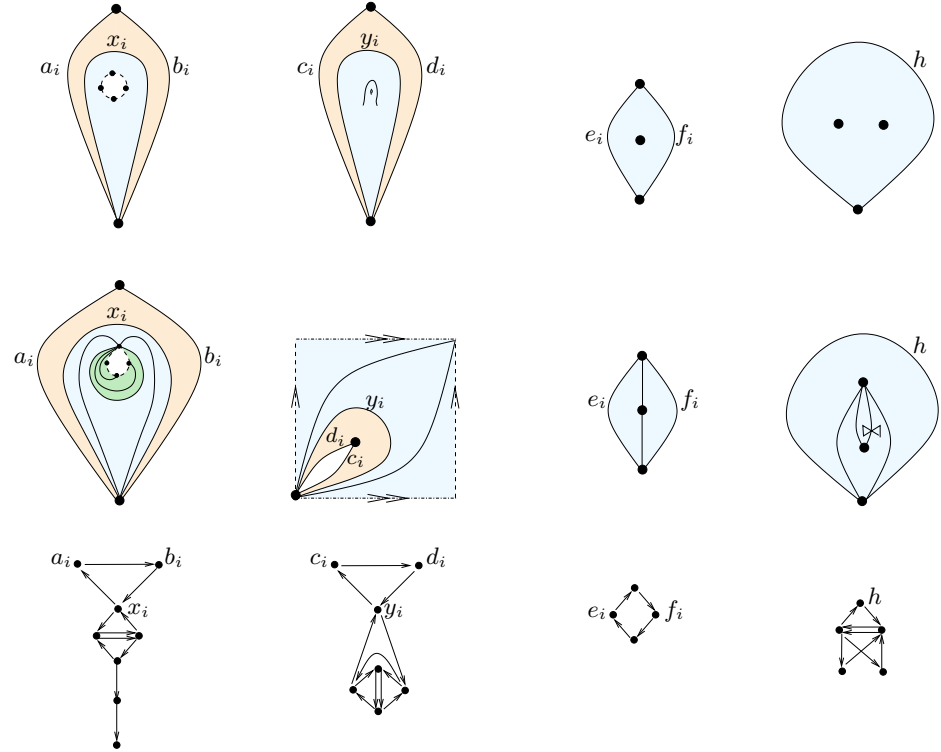,width=0.75\linewidth}
\put(-315,266){\scriptsize $x_i$ }  
\put(-340,256){\scriptsize $a_i$ }  
\put(-290,256){\scriptsize $b_i$ }  
\put(-219,267){\scriptsize $y_i$ }  
\put(-243,256){\scriptsize $c_i$ }  
\put(-194,256){\scriptsize $d_i$ }  
\put(-134,230){\scriptsize $e_i$ }  
\put(-100,230){\scriptsize $f_i$ }  
\put(-10,258){\scriptsize $h$ }  
\put(-315,163){\scriptsize $x_i$ }  
\put(-350,146){\scriptsize $a_i$ }  
\put(-278,146){\scriptsize $b_i$ }  
\put(-217,130){\scriptsize $y_i$ }
\put(-133,130){\scriptsize $e_i$ }  
\put(-98,130){\scriptsize $f_i$ }
\put(-10,150){\scriptsize $h$ }  
\put(-224,108){\tiny $c_i$ }  
\put(-232,116){\tiny $d_i$ }  
\put(-307,51){\scriptsize $x_i$ }  
\put(-336,70){\scriptsize $a_i$ }  
\put(-293,70){\scriptsize $b_i$ }  
\put(-210,51){\scriptsize $y_i$ }  
\put(-133,50){\scriptsize $e_i$ }  
\put(-104,50){\scriptsize $f_i$ }
\put(-27,55){\scriptsize $h$ }  
\put(-238,70){\scriptsize $c_i$ }  
\put(-194,70){\scriptsize $d_i$ }  
\caption{Features (top row), their standard triangulations (middle row) and corresponding quivers (bottom). Columns from left to right: a digon with a hole, a digon with a handle, a digon with a puncture, a monogon with two punctures.}
\label{parts}
\end{center}
\end{figure}

The quiver $Q$ corresponding to the standard  triangulation is shown in Fig.~\ref{st_q}.
It consists of the following elements built into a chain (from the right to the left): 
\begin{itemize}
\item[-] quiver $Q_{out}$ of triangulated outer polygon $S_{out}$; 
\item[-] quivers of  digons with  punctures;
\item[-] quivers of  digons with  handles;
\item[-] quivers of  digons with  holes;
\item[-] in case of absence of holes and handles, the leftmost element will be the quiver of a monogon with two punctures.   
\end{itemize}

\begin{figure}[!h]
\begin{center}
  \epsfig{file=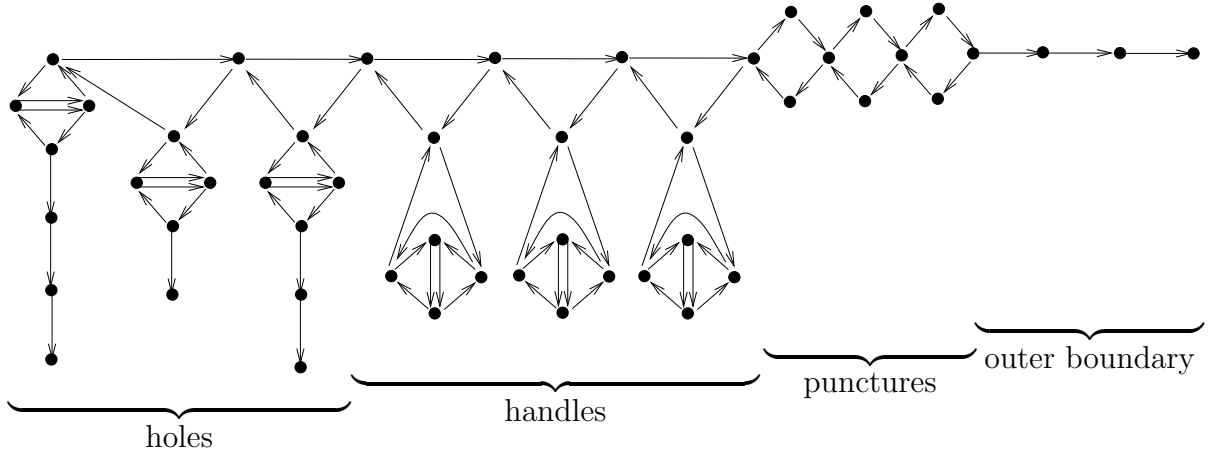,width=0.99\linewidth}
  \put(-89,20){  $ \underbrace{\phantom{aaaaaaaaaaaaaa}}_{\text{\normalsize outer boundary}}   $    }  
  \put(-169,10){  $ \underbrace{\phantom{aaaaaaaaaaaaa}}_{ \text{\normalsize punctures}}   $    }  
  \put(-325,-0){  $ \underbrace{\phantom{aaaaaaaaaaaaaaaaaaaaaaaaa}}_{\text{\normalsize handles}}   $    }  
  \put(-455,-10){  $ \underbrace{\phantom{aaaaaaaaaaaaaaaaaaaaa}}_{\text{\normalsize holes}}   $    }  
  
\caption{Quiver from standard triangulation}
\label{st_q}
\end{center}
\end{figure}

\begin{notation}
\label{not}  
We will highlight the following subquivers of $Q$, as in Fig.~\ref{q_not}:
\begin{itemize}
\item[-] $Q_{out}$: the subquiver of the outer polygon $S_{out}$ (if the outer component contains other marked points than $p$);
\item[-] two vertices, $v_1$ and $v_2$, connected to $Q_{out}$ (see Fig.~\ref{q_end} showing $v_1$ and $v_2$ depending on whether $S_{out}$ is empty and whether the first feature from the right is a hole, a handle or a puncture),
  the arcs corresponding to  $v_1$ and $v_2$ will be denoted by  $\gamma_1$ and $\gamma_2$;
  
\item[-] $Q_{in}$: the subquiver corresponding to the inner boundary components, i.e. $Q_{in}$ is spanned by all vertices corresponding to arcs of the triangulation with at least one endpoint on any of inner boundary components; 
\item[-] the subquiver $Q_I$ spanned by all other vertices of $Q$, where $I$ is the index set of vertices not lying in $Q_{out}$, $Q_{in}$ and different from $v_1$ and $v_2$.

\end{itemize}

\end{notation}

\begin{figure}[!h]
\begin{center}
  \epsfig{file=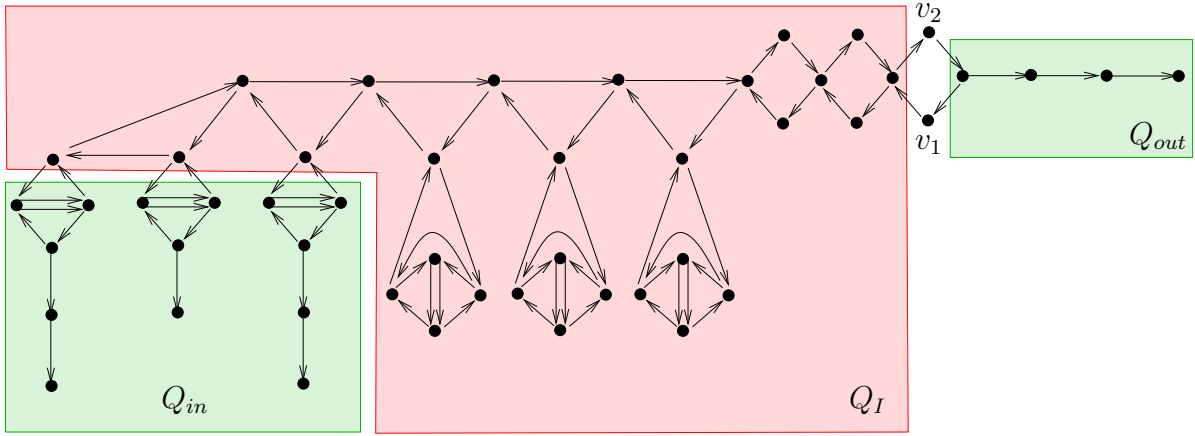,width=0.99\linewidth}
  \put(-29,110){ $Q_{out}$    }  
  \put(-110,108){ $v_1$    }  
 \put(-110,158){ $v_2$    }  
  \put(-135,10){ $Q_I$   }  
  \put(-395,10){ $Q_{in}$   }    
\caption{Notation: subquivers of the quiver for standard triangulation.}
\label{q_not}
\end{center}
\end{figure}

\begin{figure}[!h]
\begin{center}
  \epsfig{file=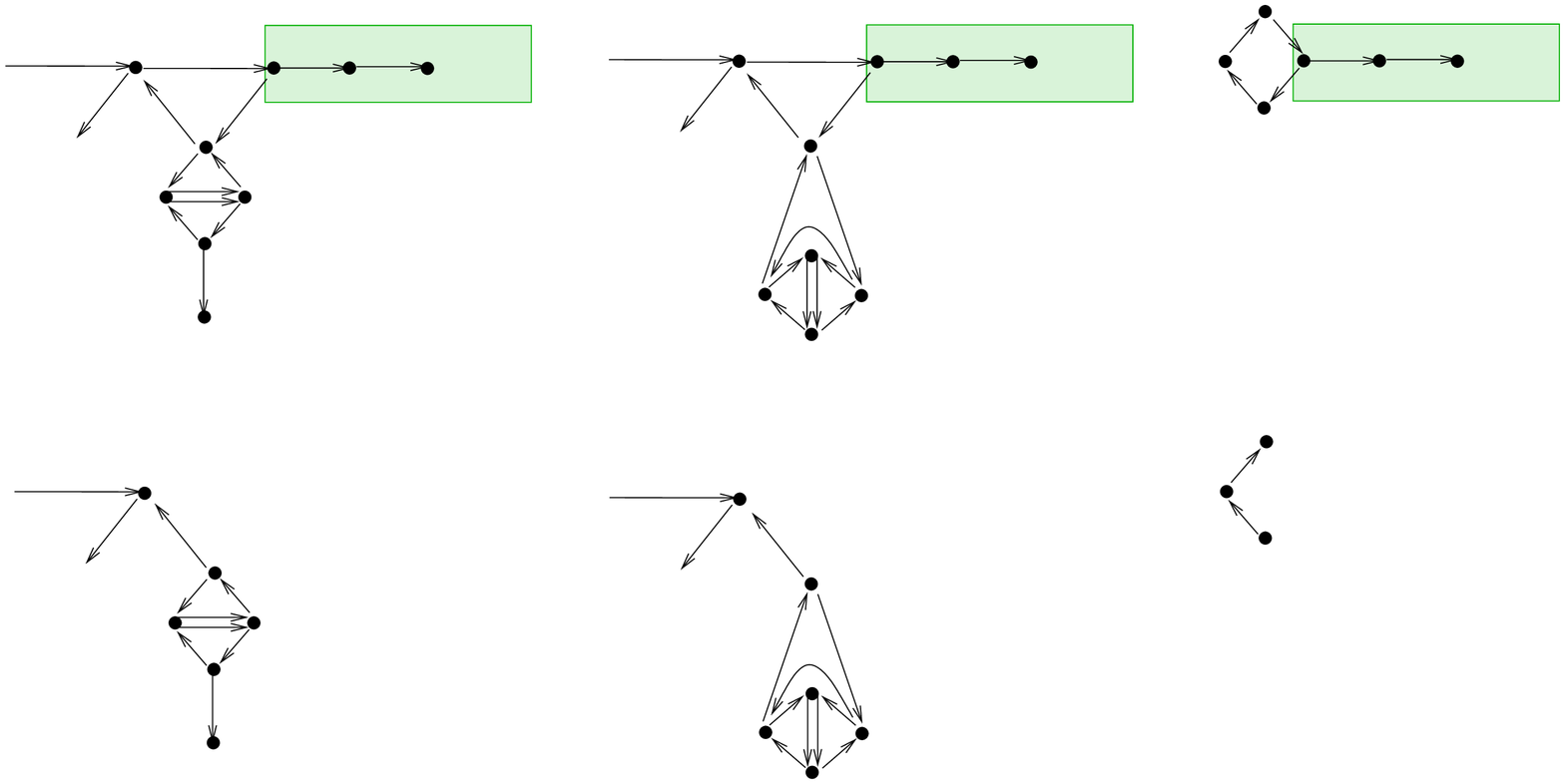,width=0.99\linewidth}
  \put(-449,2){ (a)    }  
  \put(-279,2){ (b)    }  
  \put(-109,2){ (c)    }  
  \put(-28,200){ $Q_{out}$    }  
   \put(-150,200){ $Q_{out}$    }  
   \put(-325,200){ $Q_{out}$    }  
  \put(-100,225){ $v_2$    }  
 \put(-100,185){ $v_1$    }  
  \put(-102,101){ $v_2$    }  
 \put(-100,61){ $v_1$    }  
  \put(-238,81){ $v_2$    }  
 \put(-219,58){ $v_1$    }  
  \put(-243,212){ $v_2$    }  
 \put(-218,178){ $v_1$    }  
  \put(-410,81){ $v_2$    }  
 \put(-391,58){ $v_1$    }  
  \put(-420,210){ $v_2$    }  
 \put(-391,178){ $v_1$    }  
 \caption{Vertices $v_1$ and $v_2$ for the cases when the rightmost feature is a hole (a), a handle (b), or a puncture (c), drawn for the case with $S_{out}\ne \emptyset$ (above) and for  $S_{out}= \emptyset$ (below).}
\label{q_end}
\end{center}
\end{figure}

\begin{theorem}
  \label{t_standard}
  Let  $S$ be a surface with at least one boundary component distinct from a  disk with at most two punctures and from an unpunctured annulus. 
  Suppose that $S$ is triangulated in the standard way.
  Then a coefficient vector $\b=(b_1,\dots,b_n)$ is admissible if and only if it satisfies the following conditions:
 \begin{itemize} 
\item[(a1)] $b_i=0$ for  $i\in I$; 
\item[(a2)] the annulus property is satisfied;
\item[(a3)] for the vertices $v_1$ and $v_2$ one has $b_1=-b_2\le 0$.  
  
 \end{itemize}  
\end{theorem}

To prove the theorem we will use the following terminology.

\begin{definition}
  Let $L$ be a lamination and $C\in L$ be a curve. Let $T$ be a triangulation. Then crossings of arcs of $T$ with $C$ cut $C$ into {\it subsegments}, and by a {\it segment} we  mean any connected union of subsegments of $C$ (with respect to $T$). Two consecutive subsegments form a {\it crossing} with $T$. A crossing is {\it non-trivial} if its input into   shear coordinates of $L$ is non-zero, otherwise it is {\it trivial}. In the latter case, both subsegments can be isotopically deformed to be contained in a small neighborhood of the same vertex $q$ of the corresponding quadrilateral, and will be called {\it $q$-local}. The crossing formed by two $q$-local subsegments will be  called $q$-local, as well as  any segment formed of $q$-local subsegments.
  
\end{definition}  

We make the following elementary observation:

\begin{prop}
\label{none}  
  Let $T$ be a triangulation of a marked surface and $L$ be a lamination. Choose $\gamma_i \in T$, and suppose that there exists a non-trivial crossing of $\gamma_i$ with a curve $C\in L$.  Then $b_i(L)\ne 0$ and $\sgn(b_i(C))=\sgn(b_i(L))$.

\end{prop}

The proof immediately follows from the definition of shear coordinates: segments inducing crossings of different signs inside a quadrilateral with diagonal $\gamma_i$ intersect each other. The case of self-folded triangles is treated similarly.

\begin{proof}[Proof of Theorem~\ref{t_standard}]
 In view of Theorem~\ref{t_per}, we need to show that the conditions in the theorem hold if and only if the lamination is peripheral. The plan of the proof will be similar to the one of Lemma~\ref{l_ann}.  

 We will consider the arcs
 $\gamma_1$ and $\gamma_2$  corresponding to vertices $v_1$ and $v_2$ defined as shown in Fig.~\ref{q_end}.
 These arcs look as in Fig.~\ref{arc12} depending on the presence of punctures in $S$.

\begin{figure}[!h]
\begin{center}
  \epsfig{file=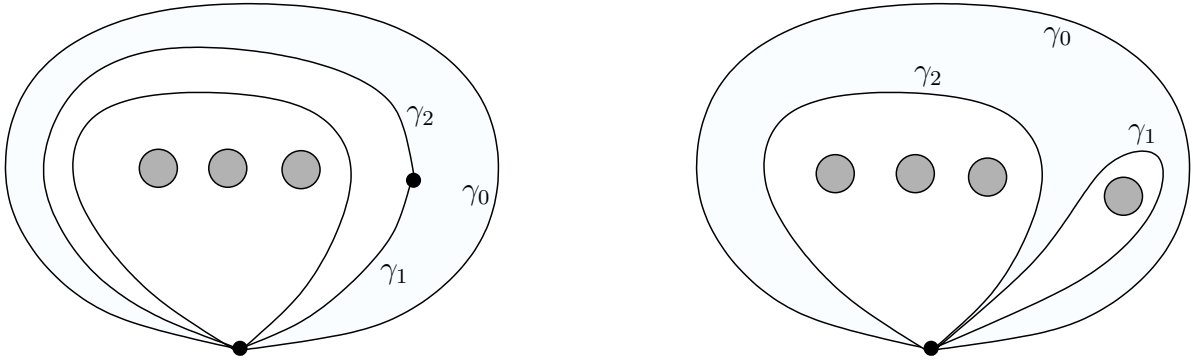,width=0.99\linewidth}
  \put(-61,120){ $\gamma_0$  }  
  \put(-29,83){ $\gamma_1$    }  
  \put(-110,105){ $\gamma_2$    }    
  \put(-281,60){ $\gamma_0$  }  
 \put(-312,30){ $\gamma_1$  }  
  \put(-302,90){ $\gamma_2$   }    
\caption{Arcs $\gamma_1$ and $\gamma_2$ in case of no punctures in $S$ (right) and otherwise (left). The grey circles indicate features (distinct from punctures on the right). If the outer boundary component contains a unique marked point, the arc $\gamma_0$ coincides with the outer boundary.  Arcs $\gamma_0,\gamma_1$ and $\gamma_2$ form one triangle of the triangulation. }
\label{arc12}
\end{center}
\end{figure}

 \medskip
 \noindent
 {\bf Conditions (a1)--(a3) are necessary.}
 We need to show that if $L$ is peripheral then  (a1)--(a3) hold.

 We start by  proving (a3).
 Let $L$ be a peripheral lamination. 
 Notice that any peripheral curve homotopic to an inner boundary does not cross $\gamma_1$ and $\gamma_2$. Consider peripheral curves homotopic to the outer boundary. 
 
 First, consider the closed curve $C$ homotopic to the outer boundary, see Fig.~\ref{closed}. It is easy to see that for this curve $b_1=-b_2=-1$. 
 Furthermore, any non-closed peripheral curve has $b_1=b_2=0$ (it either does not cross  $\gamma_1$ and $\gamma_2$ at all, or consequently crosses $p$-locally all curves incident to $p$).
 So, no peripheral curve except for $C$ can affect $b_1$ and $b_2$, and hence  condition (a3) is necessary.

Condition (a2) is necessary in view of Corollary~\ref{cor-ann}. Condition (a1) is necessary since no peripheral curve crosses non-trivially any arc of the triangulation corresponding to any vertex of $Q_I$.
Hence,  conditions (a1)--(a3) are necessary.

\begin{figure}[!h]
\begin{center}
  \epsfig{file=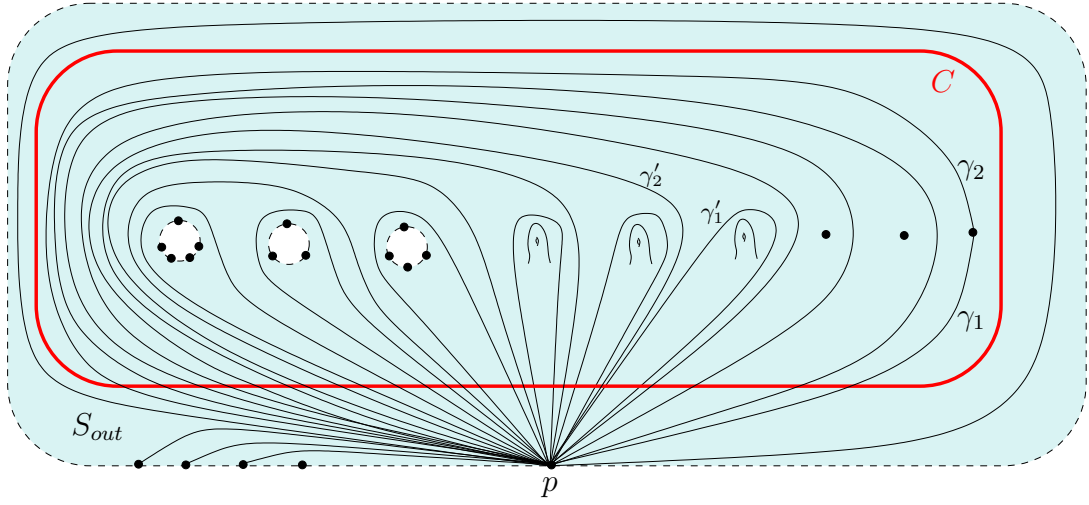,width=0.9\linewidth}
  \put(-385,13){$S_{out}$}
  \put(-60,143){\color{red} $C$}
    \put(-207,-8){$p$}
\put(-50,112){$\gamma_2$}
  \put(-50,55){$\gamma_1$}
\put(-170,109){\scriptsize $\gamma_2'$}
  \put(-147,95){\scriptsize $\gamma_1'$}
\caption{The closed curve $C$ isotopic to the outer boundary component. The curves $\gamma_i'$ play the role of $\gamma_i$ in case of absence of punctures.}
\label{closed}
\end{center}
\end{figure}

 \medskip
 \noindent
 {\bf Conditions (a1)--(a3) are sufficient.}
Next, we will prove that every lamination which is not peripheral contradicts some of conditions (a1)--(a3).

Suppose that $L$ is a non-peripheral lamination and suppose that all conditions  (a1)--(a3) are satisfied by $L$.
Let us make several observations:

\begin{itemize}
    \setlength\itemsep{15pt}
\item[(O1) $\bullet$]
  {\it No curve from $L$  has any end on any inner boundary component except for the peripheral curves. No curve from $L$ can spiral into a puncture.
In particular, every curve $l\in L$ consists of finitely many subsegments.}

\smallskip    
    The first statement follows from condition (a2), the argument goes along the same lines as the part of the proof of  Lemma~\ref{l_ann} concerning bridging arcs. The second statement follows from the fact that a spiralling curve produces a non-zero shear coordinate on one of the two arcs incident to the puncture (see~\cite[Fig. 36]{FT} and~\cite[Fig. 6.3]{FeSTu3}) and from Proposition~\ref{none}.

\item[(O2)  $\bullet$] 
 {\it Let $x_i$ be an arc of $T$ with both ends at $p$ and enclosing exactly one inner boundary component, see Fig.~\ref{parts}. Then for any curve  $c\in L$ intersecting $x_i$
    the restriction of $c$ onto the annulus cut out by $x_i$ is a $p$-local segment of $c$}.

\smallskip 
  The statement follows immediately from (O1).

\item[(O3)  $\bullet$] 
 {\it Let $l\in L$, and let $\gamma\in T$ be incident to $p$ and encircled by $\gamma_1$, $\gamma_2$, or $\gamma_1\cup\gamma_2$. Then every intersection of $l$ with $\gamma$ is $p$-local}.

 \smallskip 
For arcs inside $x_i$ this follows from (O2); all other arcs incident to $p$  and encircled by $\gamma_1$,  $\gamma_2$, or  $\gamma_1\cup\gamma_2$ correspond to vertices of $Q$  belonging to $Q_I$, therefore the statement follows from Proposition~\ref{none} together with (O2).

\item[(O4)  $\bullet$]
  {\it Let $l\in L$. Then every intersection of $l$ with $\gamma_i$ belongs to a $p$-local segment of $l$ with two endpoints either on $\gamma_i$ (if $\gamma_i$ is a loop), or on $\gamma_1\cup\gamma_2$ (otherwise).}

  \smallskip 
According to (O3), all subsegments of $l$ inside a monogon bounded by $\gamma_i$ (or the digon bounded by $\gamma_1\cup\gamma_2$) are $p$-local, so they compose a $p$-local segment. Due to (O1), $l$ is either closed or have both ends on the outer boundary component. Therefore, every maximal segment of $l$ contained in  $\gamma_i$ (or in $\gamma_1\cup\gamma_2$) has both ends on $\gamma_i$ (or on $\gamma_1\cup\gamma_2$, respectively).

\item[(O5) $\bullet$]
  {\it Suppose that $\gamma_1$ and $\gamma_2$ are arcs with one endpoint in a puncture, as in Fig.~\ref{arc12}, left.
    Let $l\in L$ and suppose that $b_1(l)\ne 0$ or $b_2(l)\ne 0$. Then $l$ coincides with the closed curve $C$ (see Fig.~\ref{closed}).
  }

  \smallskip
  Suppose that $b_1(l)\ne 0$ (the case of $b_2(l)\ne 0$ can be treated similarly). Let $t_0$ be an intersection point of $l$ and $\gamma_1$ producing a non-trivial crossing. By (O4) there is a $p$-local segment $t_0t_1$ in $l$ with $t_1\in \gamma_1\cup \gamma_2$, more precisely, $t_1\in \gamma_2$, see Fig.~\ref{spiral}. Since the crossing at $t_0$ is non-trivial, the subsegment $t_{-1}t_0$ of $l$  not lying on $t_0t_1$  should have its end $t_{-1}$ on $\gamma_2$. If $t_{-1}=t_1$ then $l$ is the closed curve $C$.

  Suppose that $t_{1}$ lies on $\gamma_2$ further from $p$ than $t_{-1}$.  Extending the segment $t_{-1}t_1\in l$ past $t_1$ we will obtain a point $t_2$ on $\gamma_1$ lying further away from $p$ than $t_0$. By (O4), there  is a $p-$local segment $t_2t_3$ of $l$ with $t_3\in\gamma_2$. Notice that we will get  $t_3$ further away from $p$ then $t_1$. Continuing in the same way we will get infinitely many subsegments of $l$ in contradiction to (O1).  The case when $t_1$ lies on $\gamma_2$ closer to $p$ than $t_{-1}$ can be treated similarly (by extending the curve past $t_{-1}$).

\begin{figure}[!h]
\begin{center}
  \epsfig{file=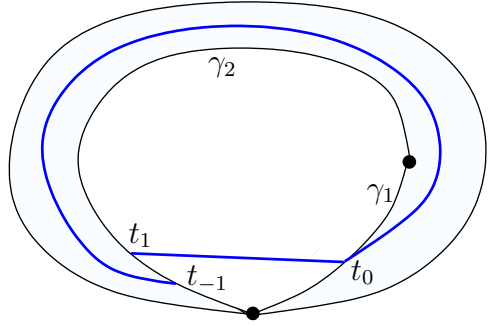,width=0.4\linewidth}
  \put(-56,16){ $t_0$  }  
  \put(-140,30){ $t_1$  }  
  \put(-118,14){ $t_{-1}$  }  
  \put(-50,47){ $\gamma_1$    }  
  \put(-110,95){ $\gamma_2$    }    
\caption{To the proof of (O5).}
\label{spiral}
\end{center}
\end{figure}

\item[(O6)  $\bullet$]
  {\it  Suppose that $\gamma_1$ and $\gamma_2$ are loops with both endpoints in $p$, as in Fig.~\ref{arc12}, right.
    Denote by $l^\pm_1$ and $l^\pm_2$  the positive and negative elementary laminations  for $\gamma_1$ and $\gamma_2$ respectively, and by $D^r_C(l^\pm_i)$ twists along $C$ applied to the curves above, $i=1,2$, $r\in\Z$. Denote by $M$ the set of curves consisting of the closed curve $C$ and the curves whose restriction onto $S\setminus S_{out}$  coincides with the restrictions of curves $l_1^+$, $l_2^-$, $D^r_C(l^+_i)$, $D^{-r}_C(l^-_i)$, where $r>0$, $i=1,2$. Then if $l\in L$ and $l\notin M$, then $b_1(l)=b_2(l)=0$. } 

\smallskip 
Let $l\in L$ be a curve, and suppose that at least one of $b_1(l)$ and $b_2(l)$ is not zero. This implies that $l$ intersects at least one of $\gamma_1$ and $\gamma_2$. Notice that  $S\setminus S_{out}$ consists of one triangle bounded by $\gamma_0,\gamma_1,\gamma_2$ and two surfaces encircled by $\gamma_1$ and $\gamma_2$ respectively (here $\gamma_0$ may coincide with the outer boundary). In view of (O4), the segments of $l$ contained inside the arcs $\gamma_1$ and $\gamma_2$ are $p$-local, and thus uniquely determined, see Fig.~\ref{almost_cl}, left. We now want to list all possible subsegments of $l$ inside the remaining triangle with two ends on $\gamma_1$ and $\gamma_2$.

\begin{figure}[!h]
\begin{center}
  \epsfig{file=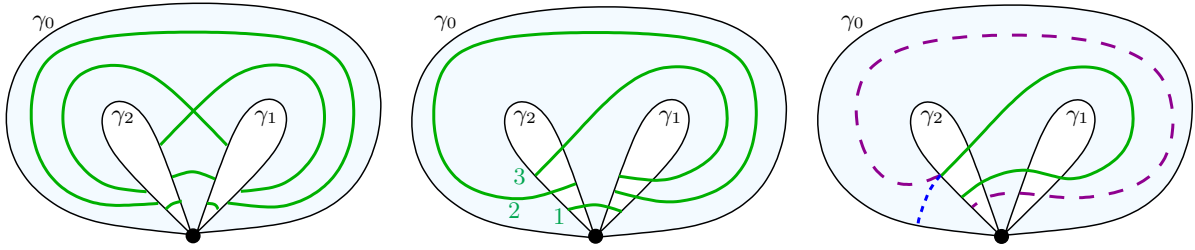,width=0.99\linewidth}
 \put(-243,8){\scriptsize \color{Green} $1$}
 \put(-260,10){\scriptsize \color{Green} $2$}
 \put(-258,23){\scriptsize \color{Green} $3$}
 \put(-134,83){\scriptsize  $\gamma_0$}
 \put(-288,83){\scriptsize  $\gamma_0$}
 \put(-440,83){\scriptsize  $\gamma_0$} 
 \put(-104,47){\scriptsize  $\gamma_2$}
 \put(-49,47){\scriptsize  $\gamma_1$}
 \put(-258,47){\scriptsize  $\gamma_2$}
 \put(-203,47){\scriptsize  $\gamma_1$}
 \put(-410,47){\scriptsize  $\gamma_2$}
 \put(-356,47){\scriptsize  $\gamma_1$}
\caption{To the proof of (O5): behaviour of curves on $S\setminus S_{out}$ }
\label{almost_cl}
\end{center}
\end{figure}

We say that a subsegment in the triangle joining $\gamma_1$ and $\gamma_2$ {\em approaches $\gamma_i$ from the right (left)} if it is followed by a $p$-local segment whose other end can be reached by going around $p$ counterclockwise (resp., clockwise). Since every subsegment joining $\gamma_1$ and $\gamma_2$ approaches them from one of the two sides, there are precisely four types of subsegments, they all are shown on Fig.~\ref{almost_cl}, left.

Notice that two of the four subsegments intersect, which means that at most one of them can be a part of $l$; we assume first that the one approaching both curves from the left does not appear. Gluing the $p$-local segments located inside $\gamma_1$ and $\gamma_2$ to all three remaining subsegments, we conclude that $l$ can be assembled from the copies of the three curves shown in Fig.~\ref{almost_cl}, middle; we will refer to these as to segments of types $1, 2, 3$ respectively.  These segments are attached to each other in $l$ along $p$-local segments with both ends on the same curve $\gamma_i$.

Suppose that $l$ does not contain any segment of type 3. It is easy to see that in this case none of the segments can be extended to an intersection with $\gamma_0$ (except for a $p$-local extension of a type 1 segment which has $b_1=b_2=0$ and thus is excluded), which means that $l$ is a closed curve. The only non-self-intersecting closed curve that can be composed out of segments of types 1 and 2 is the closed curve $C$.

Suppose now that $l$ contains a segment of type 3. It can only be extended past its intersection with $\gamma_1$ by a type 1 segment, see Fig.~\ref{almost_cl}, right. We obtain a segment with both ends on $\gamma_2$. As it contains two $p$-local segments located inside $\gamma_2$, we can determine which of the ends is closer to $p$ along $\gamma_2$, call it the {\it lower} end and the other one the {\it upper} end.   Now,  the upper end can be either joined to $\gamma_0$ or extended by a type 2 segment. Notice that if it is joined to $\gamma_0$, then the lower end should also be joined to $\gamma_0$, and thus we obtain a restriction of $l_1^+$ onto $S\setminus S_{out}$. If the upper end is extended using a segment of type 2, then we obtain a new curve with both ends on $\gamma_1$ and well defined upper and lower ends, so we can repeat the reasoning for the new upper end. We will need to connect the upper end to the boundary after finitely many steps (as $l$ consists of finitely many of these segments). This will result in a restriction of $D_C^{r}(l^+_i)$, $i=1,2$, $r>0$.

Finally, if while considering the four subsegments in the triangle we avoid the one approaching both curves from the right, then using precisely the same arguments we would obtain restrictions of curves $l_2^-$, $D_C^{r}(l^-_i)$ with  $i=1,2$, $r<0$.

\item[(O7)  $\bullet$]
  {\it Let $l\in L$ and $b_1(l)\ne 0$ or $b_2(l)\ne 0$. Then $l$ coincides with the closed curve $C$.}

  \smallskip
If $\gamma_1$ and $\gamma_2$ are arcs incident to a puncture as in Fig.~\ref{arc12}, left, then the statement follows immediately from (O5), so we may assume that $\gamma_1$ and $\gamma_2$ are loops as  in Fig.~\ref{arc12}, right.
  
Due to (O6), we need to consider the curves belonging to the set $M$ only. Notice that any twist $D^{k}_C(l^+_i)$, $i=1,2$, $k\ge 0$ is not compatible with any twist $D^{m}_C(l^-_j)$, $j=1,2$, $m<0$, since they contain intersecting subsegments (see Fig.~\ref{almost_cl}, left). Now, the negative shear coordinates $(b_1,b_2)$ for $D^{k_1}_C(l^+_1)$ and  $D^{k_2}_C(l^+_2)$ are equal to $(-(2k_1+1),2k_1)$ and $(-2k_2, 2k_2-1)$ respectively. According to (a3) and Prop.~\ref{none}, $k_1\ge 0$ and $k_2\ge 1$. It is easy to see that for any such curve the modulus of $b_1$ is strictly greater than the modulus of $b_2$. For $D^{m}_C(l^-_j)$ the considerations are similar. For $C$, $|b_1(C)|=|b_2(C)|$. Therefore, if $L$ contains any curve from $M$ except for $C$, then $|b_1(L)|\ne|b_2(L)|$ in contradiction to (a3).

\end{itemize}

 Recall that $L$ is a non-peripheral lamination. Let $C_{np}\in L$ be a non-peripheral curve and consider a lamination consisting  of the single curve $C_{np}$ (we will use the same notation for this lamination).
 We now show that there exists a non-peripheral curve which coincides with $C_{np}$ inside $S\setminus S_{out}$ and has all shear coordinates equal to 0 in contradiction to~\cite{FT}.
 
 Recall from Notation~\ref{not} that $Q_{in}$, $Q_{out}$ and $Q_I$ are subquivers of $Q$ corresponding to inner boundary, outer boundary and the set defined in Fig.~\ref{q_not}. Denote by $I_{in}$ and $I_{out}$ the corresponding index sets. Denote also $Q_{12}=\langle v_1,v_2\rangle$.
 
 Observe:

\begin{itemize} 
\item[$\bullet$] $b_i(C_{np})=0$ for $i\in I$ (by Proposition~\ref{none} and (a1)); 

 \item[$\bullet$] $b_i(C_{np})=0$ for $i=1,2$.\\
  This follows immediately from Observation (O7) since $C_{np}$ is non-peripheral and hence does not coincide with $C$.

\item[$\bullet$]  $b_i(C_{np})=0$ for $i\in I_{in}$. %, where  $I_{in}=\cup_{j=1}^k I_{in_j}$}
    \\
This follows from applying observation (O2) to each inner boundary component. % $I_{in_j}$.       

\end{itemize}

Therefore, we obtain that
\begin{itemize} 
\item[$\bullet$] $b_i(C_{np})=0$ for $i\notin I_{out}$. 

\end{itemize}

We are left to consider $b_i(C_{np})$ for $i\in I_{out}$ (notice that this only makes sense when $S_{out}$ is non-empty).
If ends of the curve $C_{np}$ do not lie on the outer boundary component (i.e., $C_{np}$ is closed), then $C_{np}$ does not cross any arc corresponding to vertices of $I_{out}$ and we have  $b_i(C_{np})=0$ for $i\in I_{out}$. In this case all shear coordinates of $C_{np}$ vanish, which contradicts~\cite{FT}. Thus, we can assume that $C_{np}$ has both ends on the outer boundary. 

We  will now modify the curve  $C_{np}$ by amending its intersection with the subsurface $S_{out}$ only. 
The new curve $C_{np}'$ is defined by shifting  each endpoint of $C_{np}$ to one of the boundary intervals  containing the marked point $p$ according to the following rules: the ends of segments crossing consequently  $\gamma_1$ and $\gamma_0$ will be shifted clockwise along the outer boundary, and the ends of segments crossing consequently  $\gamma_2$ and $\gamma_0$ will be shifted counterclockwise, see   Fig.~\ref{shift}. As a result, all crossings of $C_{np}'$  with arcs in $S_{out}$ are  $p$-local (including the crossings with $\gamma_0$), and hence we get $b_i(C_{np}')=0$ for  $i\in I_{out}$.
As we also have  $b_i(C_{np}')=0$ for  $i\notin I_{out}$, we conclude that all shear coordinates of $C_{np}'$ vanish, which leads to a contradiction.

This  shows that non-peripheral lamination $L$ satisfying (a1)--(a3) does not exist, which proves that the conditions (a1)--(a3) are sufficient.

\begin{figure}[!h]
\begin{center}
  \epsfig{file=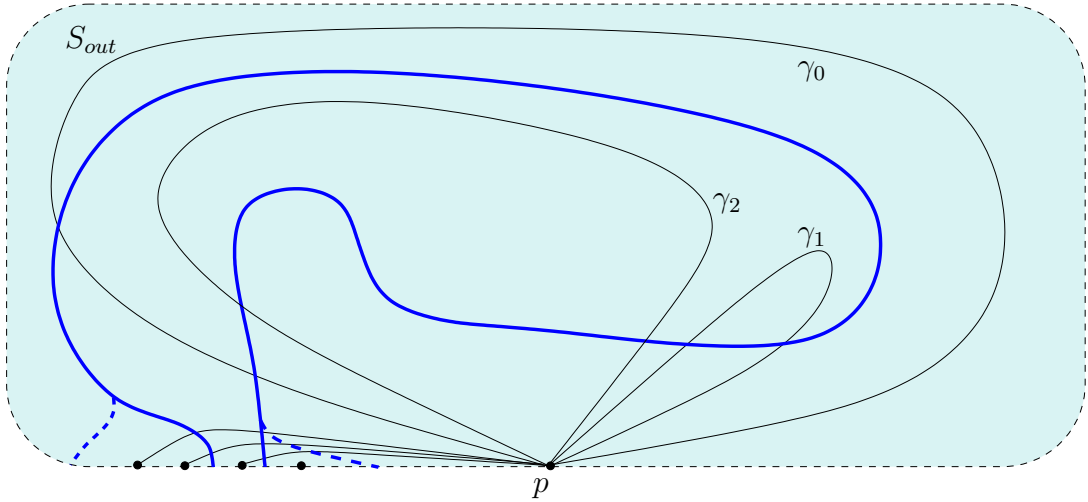,width=0.9\linewidth}
    \put(-387,160){$S_{out}$}
    \put(-210,-8){$p$}
    \put(-110,150){$\gamma_0$}
  \put(-110,87){$\gamma_1$}
  \put(-142,100){$\gamma_2$}
\caption{Shifting endpoints of the curve $C_{np}$ on the outer boundary (to the segment on the left of $p$ if the curve  come to $\gamma_0$  from $\gamma_2$, and to the segment on the right of $p$ is it comes from $\gamma_1$}
\label{shift}
\end{center}
\end{figure}
\end{proof}

\section{Skew-symmetrizable mutation classes}
\label{symmetrizable}
\noindent
In this section we consider the skew-symmetrizable case.

Let $B$ be a skew-symmetrizable $n\times n$ matrix, i.e. there is an integer diagonal $n\times n$ matrix $D=(d_i)$ with positive entries such that $BD$ is skew-symmetric. We suppose that $B$ is mutation-finite and want to determine whether $B$ can be complemented by one more row $(b_{n+1,1}, \dots, b_{n+1,n})$ so that the obtained $(n+1)\times n$ matrix $\widetilde B$ will be still mutation-finite. As before, we call a vector $\bm b=(b_{1}, \dots, b_{n})$ admissible if the matrix $\widetilde B$ composed of $B$ and row $-\bm b$ is mutation-finite.

\subsection{Diagrams and unfoldings}

We recall basics on diagrams of skew-symmetrizable matrices.

\medskip
\noindent
{\bf Diagrams.}
According to~\cite{FZ2}, skew-symmetrizable matrices $(b_{ij})$ can be represented by diagrams with arrows from $v_i$ to $v_j$ of weight $-\sgn{(b_{ij})}\,b_{ij}b_{ji}$, which undergo mutations compatible with matrix mutations. A skew-symmetrizable matrix $(b_{ij})$ can be reconstructed by its diagram and the diagonal skew-symmetrizing matrix $D=(d_i)$. We will use a double arrow   $i\!=\!\!>\!j$ to denote an arrow of weight $4$ when $d_i=d_j$. 

Notice that if $B$ is skew-symmetrizable with the skew-symmetrizer $D=(d_i)$ then the  $(n+1)\times n$ matrix $\widetilde B$ can  always be extended to a skew-symmetrizable $(n+1)\times(n+1)$ matrix by adding $(n+1)$st column satisfying $b_{i,n+1}=-d_ib_{n+1,i}$ and setting $d_{n+1}=1$. This means that the matrix  $\widetilde B$ can also be represented by a diagram (with arrows of weight $\sgn(b_i)d_ib_i^2$ from $v_i$ to the frozen vertex $v_{n+1}$).

One diagram with a frozen vertex may correspond to several skew-symmetrizable extended matrices, however, for any $k=1,\dots,n$ mutations $\mu_k$ of such matrices always lead to the same extended diagram. We will call a diagram with a frozen vertex {\em mutation-finite} if it represents mutation-finite matrices (with respect to mutations in the first $n$ vertices).

\medskip
\noindent
{\bf Mutation-finite diagrams without frozen vertices.}

It was shown in~\cite{FeSTu2,FeSTu3} that mutation-finite diagrams either are skew-symmetric, or arise from triangulated orbifolds, or are of rank 2, or are mutation-equivalent to one of the seven types $F_4, \widetilde G_2, \widetilde F_4, G_2^{(*,+)},  G_2^{(*,*)}, F_4^{(*,+)},  F_4^{(*,*)}$
shown in Fig.~\ref{fg}. 

We will consider the orbifolds, rank 2 diagrams, and each of the seven exceptional mutation-finite types separately, mostly based on the notion of {\em unfolding}.

\begin{figure}[!h]
\begin{center}
  \epsfig{file=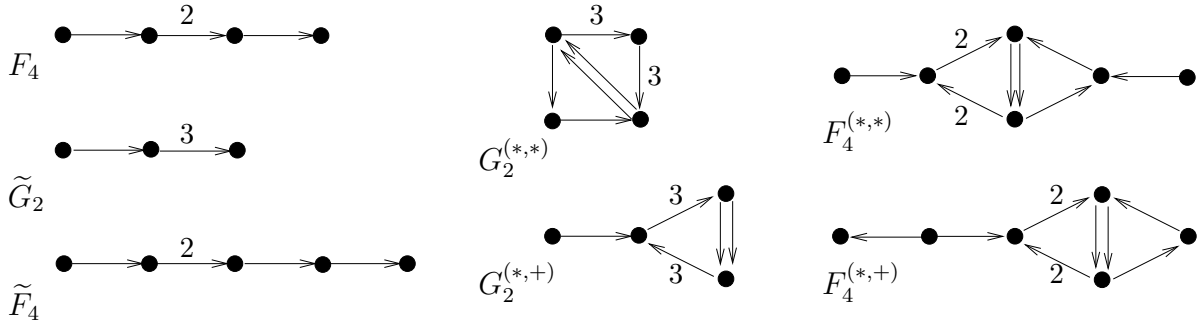,width=0.95\linewidth}
  \put(-450,81){$F_4$}
  \put(-385,100){\small $2$}
  \put(-450,31){$\widetilde G_2$}
  \put(-385,55){\small $3$}
    \put(-450,-10){$\widetilde F_4$}
    \put(-385,12){\small $2$}
  \put(-272,45){$  G_2^{(*,*)}$}    
  \put(-272,-2){$  G_2^{(*,+)}$}    
  \put(-208,77){\small 3}
  \put(-230,100){\small 3}
  \put(-200,2){\small 3}
  \put(-200,32){\small 3}
  \put(-142,55){$  F_4^{(*,*)}$}    
  \put(-142,-2){$  F_4^{(*,+)}$}      
  \put(-92,91){\small 2}
  \put(-92,62){\small 2}
  \put(-56,2){\small 2}
  \put(-56,32){\small 2}
\caption{Diagrams of exceptional skew-symmetrizable mutation-finite types.  }
\label{fg}
\end{center}
\end{figure}

\medskip
\noindent
{\bf Unfoldings.} We briefly recall the definition of an unfolding of a skew-symmetrizable matrix introduced by A.~Zelevinsky. For more details see~\cite{FeSTu2}.

Let $B$ be a skew-symmetrizable matrix with a skew-symmetrizer $D=(d_i)$. 
Suppose that we have chosen disjoint index sets $E_1,\dots, E_k$ with $|E_i| =d_i$. Denote $m=\sum\limits_{i=1}^k d_i$.
Suppose also that we choose a skew-symmetric integer matrix $C$ of size $m\times m$ with rows and columns indexed by the union of all $E_i$, such that

(1) the sum of entries in each column of each $E_i \times E_j$ block of $C$ equals $b_{ij}$;

(2) if $b_{ij} \geq 0$ then the $E_i \times E_j$ block of $C$ has all entries non-negative.

Define a {\it composite mutation} $\h\mu_i = \prod_{\hat\imath \in E_i} \mu_{\hat\imath}$ on $C$. This mutation is well-defined, since all the mutations  $\mu_{\hat\imath}$, $\hat\imath\in E_i$, for given $i$ commute.

We say that $C$ is an {\it unfolding} of $B$ if $C$ satisfies assertions $(1)$ and $(2)$ above, and for any sequence of iterated mutations $\mu_{k_1}\dots\mu_{k_m}(B)$ the matrix $C'=\h\mu_{k_1}\dots\h\mu_{k_m}(C)$ satisfies assertions $(1)$ and $(2)$ with respect to $B'=\mu_{k_1}\dots\mu_{k_m}(B)$.

If $C$ is an unfolding of a skew-symmetrizable integer matrix $B$, it is natural to define an {\it unfolding of a diagram} of $B$ as a quiver of $C$. In general, we say that a quiver $Q$ is an unfolding of a diagram $\S$ if there exist matrices $B$ and $C$ with diagram $\S$ and quiver $Q$ respectively, and $C$ is an unfolding of $B$. Note that a diagram may have many essentially different unfoldings.

We can also define an unfolding $\t C$ of an extended skew-symmetrizable matrix $\t B$ consisting of $B$ and a row $(b_{n+1,1},\dots,b_{n+1,n})={-\bm b}=-(b_1,\dots,b_n)$ in the following way: it will consist of an unfolding $C$ of $B$ and a row vector $-\h{\bm b}$ such that the block $E_{n+1}\times E_j$ consists of $d_j$ equal entries $-b_{n+1,j}$. If we extend both matrices $\t B$ and $\t C$ with one additional column each to make them skew-symmetrizable and skew-symmetric respectively, then they will satisfy assertions $(1)$ and $(2)$ with respect to any sequence of mutations not including index $n+1$. 

This leads to a definition of an unfolding of a diagram with an additional frozen vertex. Such a diagram corresponds to an extended skew-symmetrizable matrix $\t B$, so we take an unfolding of it as defined above, add an additional column to make the obtained matrix skew-symmetric, and then take the corresponding quiver. Again, such a unfolding may not be unique.  

\begin{example}
\label{a14}
Consider the skew-symmetrizable exchange matrix $B$ shown below and its diagram

\begin{center}
    $B=\begin{pmatrix}
      0&1\\
      -4&0
\end{pmatrix}$\qquad\qquad   \epsfig{file=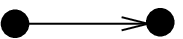,width=0.1\linewidth}
  \put(-26,-6){\small 4}   
\end{center}

We now can write the extended exchange matrix $\t B$ with a coefficient vector $(b_1,b_2)$, add a column to make it skew-symmetrizable, and draw the corresponding diagram with a frozen vertex.

\begin{center}
  $\t B= \begin{pmatrix}
      0&1\\
      -4&0\\
      {\color{blue}-b_1}& {\color{blue}-b_2}
\end{pmatrix}\quad\leadsto\quad
  \begin{pmatrix}
      0&1&{\color{magenta}b_1}\\
      -4&0&{\color{magenta}4b_2}\\
       {\color{blue}-b_1}& {\color{blue}-b_2}&{\color{magenta}0}
\end{pmatrix}\quad\leadsto \qquad\qquad$ 
\raisebox{-14pt}{\epsfig{file=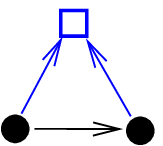,width=0.1\linewidth}
  \put(-26,-6){\small 4}
  \put(-82,16){\small \color{blue} $\sgn(b_1) b_1^2$}
  \put(-10,16){\small \color{blue}  $4 \sgn(b_2) b_2^2$}
}
\end{center}

The results of unfoldings of both the matrix and the diagram are shown below.

\begin{center}  
{\scriptsize
 $\t C= \begin{pmatrix}
      0&1&1&1&1\\
      -1&0&0&0&0\\
      -1&0&0&0&0\\
      -1&0&0&0&0\\
       -1&0&0&0&0\\
      {\color{blue}-b_1}& {\color{blue}-b_2}& {\color{blue}-b_2}& {\color{blue}-b_2}& {\color{blue}-b_2}
\end{pmatrix}\ \leadsto\ 
\begin{pmatrix}
      0&1&1&1&1&{\color{magenta}b_1}\\
      -1&0&0&0&0&{\color{magenta}b_2}\\
      -1&0&0&0&0&{\color{magenta}b_2}\\
      -1&0&0&0&0&{\color{magenta}b_2}\\
       -1&0&0&0&0&{\color{magenta}b_2}\\
      {\color{blue}-b_1}& {\color{blue}-b_2}& {\color{blue}-b_2}& {\color{blue}-b_2}& {\color{blue}-b_2}&{\color{magenta}0}
\end{pmatrix}
\quad\leadsto \quad$ 
\raisebox{-35pt}{
  \epsfig{file=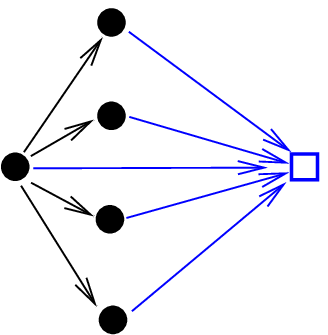,width=0.16\linewidth}
  \put(-45,40){\tiny  \color{blue} $b_1$}
  \put(-35,63){\tiny  \color{blue} $b_2$}
  \put(-35,48){\tiny  \color{blue} $b_2$}
  \put(-35,23){\tiny  \color{blue} $b_2$}
  \put(-35,7){\tiny  \color{blue} $b_2$}
}

  }
\end{center}

  \end{example}

\medskip

In general, not every skew-symmetrizable matrix admits an unfolding. However, it is shown in~\cite{FeSTu2} that every mutation-finite diagram without frozen vertices has a mutation-finite unfolding. This result provides us with a sufficient condition for a given coefficient vector to be admissible: we can always unfold a diagram with a frozen vertex to a quiver with a frozen vertex, and if the obtained quiver together with the unfolded coefficient vector is mutation-finite, then we immediately get the admissibility. Apriori, this condition is not necessary for the admissibility: mutations of a diagram correspond to a very limited collection of mutations of the unfolded quiver, so the unfolded coefficient vector might not be admissible even when the initial vector is.

\subsection{Diagrams from orbifolds and peripheral laminations}

It is shown in~\cite{FeSTu3} that the majority of skew-symmetrizable finite mutation classes originate from triangulated orbifolds. As in the surface case, coefficient vectors are in bijective correspondence with laminations, see~\cite{FeSTu3} for details. Defining peripheral laminations in exactly the same way as for surfaces, and reasoning precisely as in Section~\ref{surfaces}, we obtain a similar result.

\begin{theorem}
\label{t_orb_per}
  Let $\Sigma$ be a diagram from a triangulated orbifold $S$. Then
admissible vectors for $\Sigma$ are in bijection with peripheral laminations on $S$.

\end{theorem}

Similarly to the surface case, given a diagram from an orbifold, results of~\cite{Gu2} allow one to reconstruct a triangulation, and results of~\cite{FT,FeSTu3} allow one to reconstruct a lamination by a coefficient vector.

\subsection{Rank 2 diagrams}

\begin{theorem}
  \label{t_rk2_}
  Let $\Sigma$ be a  rank two diagram with the arrow from $v_1$ to $v_2$ of weight $a>0$. Let $\b=(b_1,b_2)$ be an integer vector. Then
  \begin{itemize}
  \item[(1)] if $a<4$ then $\b$ is admissible for any $b_1,b_2$; 
  \item[(2)] if $a=4$ then $\b$ is admissible if and only if $b_1\le 0\le b_2$ and $d_1b_1^2=d_2b_2^2$;
  \item[(3)] if $a>4$ then there are no admissible vectors. 
    
  \end{itemize}
\end{theorem}

\begin{proof}
  Part (1) concerns finite types, so it follows from~\cite{FZ2}.
  
  For part (2) there are two cases: either $d_1=d_2$, or we may assume that $d_1=1, d_2=4$. The former case is skew-symmetric and thus follows from Section~\ref{aff}. Let us now consider the latter case, the corresponding diagram with coefficient vector $\bm b=(b_1,b_2)$ is shown in Example~\ref{a14}.
  To prove the sufficiency, notice that in the assumptions (2) of the theorem the square roots of the weights of the diagram change under mutations in the same way as the weights of arrows of the quiver $v_1\!=>\!v_2$ with coefficient vector $(-2b,b)$ with $b=b_2\ge 0$, so the statement follows from Lemma~\ref{l_ann} (alternatively, one can compute directly that both mutations act on the extended exchange matrix by multiplication by the negative identity matrix).

The proof of part (3) is similar to the skew-symmetric case. After at most two mutations (and swapping the labels of $v_1$ and $v_2$ if needed) we may assume that the diagram is $v_1\stackrel{a}\rightarrow v_2$, and $b_1\le 0\le b_2$. We may also assume that $d_2b_2^2\le d_1b_1^2$ (otherwise replace $\mu_1$ with $\mu_2$ in the consideration below). Then after mutation $\mu_1$ and swapping the labels of $v_1$ and $v_2$ we will obtain a diagram with coefficient vector $(b_1',b_2')$ satisfying the same conditions and $|b_i'|>|b_i|$, $i=1,2$.  
Applying iterative mutations we can increase the components of the coefficient vector indefinitely.
\end{proof}

\subsection{Affine mutation classes}

Every exceptional mutation class of diagrams of affine type contains a representative shown in Fig.~\ref{af_non}.  Every mutation class of diagrams of affine type originating from an orbifold either contains a representative with a double arrow (we show one for every mutation class  in Fig.~\ref{af_non}), or contains a representative with a subdiagram considered in Example~\ref{a14} (see~\cite{FeSTu3}). We treat these two cases separately, see Theorems~\ref{t_af_double} and~\ref{t_af_double4}.
\begin{theorem}
   \label{t_af_double}
     Let $\Sigma$ be a diagram of type $\t G_2$, $\widetilde F_4$, $\widetilde B_n$ or $\widetilde C_n$ shown in Fig.~\ref{af_non}. A coefficient vector $\b$ is admissible if and only if it satisfies the annulus property.
\end{theorem}

\begin{proof}
The annulus property is obviously a necessary condition for admissibility of a coefficient vector for given diagrams.
To see that it is also sufficient, notice that these diagrams can be unfolded to quivers of type $\widetilde D_4$, $\widetilde E_6$, $\t D_n$ and $\t A_{n,n}$ shown in Fig.~\ref{af_e678}. The unfolded coefficient vectors still satisfy the annulus property, so every unfolded quiver with frozen vertex is mutation-finite by Theorem~\ref{t_af_e}. This implies that the initial diagrams are mutation-finite as well, so the initial coefficient vectors are admissible.
\end{proof}

\begin{figure}[!h]
  \begin{center}
{\color{white}{.}}\phantom{aaa}    \epsfig{file=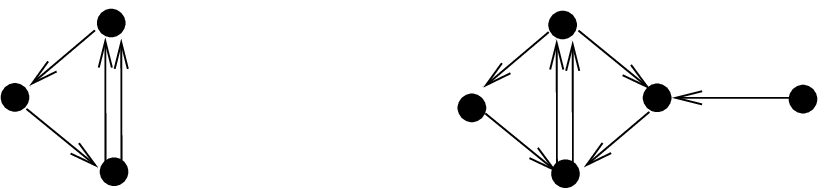,width=0.4\linewidth}
\put(-210,35){$\widetilde G_{2}$}   
\put(-173,30){\scriptsize $3$}   
\put(-173,5){\scriptsize $3$}   
\put(-105,35){$\widetilde F_4$}   
\put(-45,30){\scriptsize $2$}   
\put(-45,5){\scriptsize $2$}

\vspace{1cm}

  \epsfig{file=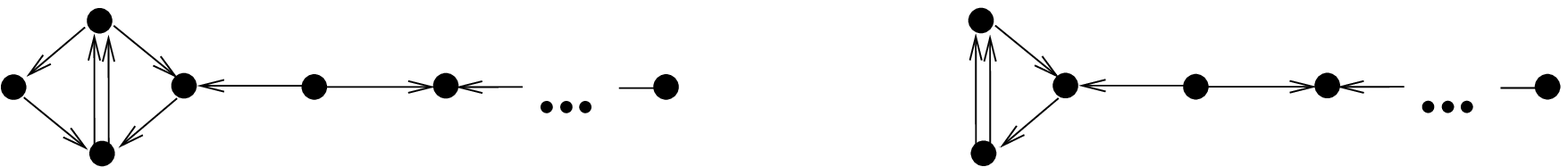,width=0.8\linewidth}
\put(-385,40){$\widetilde B_{n}$}   
\put(-357,29){\scriptsize $2$}   
\put(-357,1){\scriptsize $2$}   
\put(-165,40){$\widetilde C_n$}   
\put(-125,29){\scriptsize $2$}   
\put(-125,1){\scriptsize $2$}   
\put(-120,0){  $ \underbrace{{\phantom{aaaaaaaaaaaaaaaaaaa}}}_{\text{$n-1$}}   $    }  
\put(-327,0){  $ \underbrace{{\phantom{aaaaaaaaaaaaaaaaaaa}}}_{\text{$n-2$}}   $    }  
  \caption{Special representatives from non-skew-symmetric mutation classes of affine types. }
\label{af_non}
\end{center}
\end{figure}

\begin{remark}
  \label{aff_gen1}
By using unfoldings, we can  extend the result of Theorem~\ref{aff_gen} to a general skew-symmetrizable case, i.e. for any diagram   $\Sigma$ of affine type containing a double arrow,  a vector $\b$ is admissible if and only if $\b$ satisfies the annulus property.

\end{remark} 

In the case of a  diagram containing a subdiagram from Example~\ref{a14} we cannot use unfolding: the unfolded diagram is simply-laced, so the annulus property does not lead to any restrictions.

\begin{theorem}
   \label{t_af_double4}
     Let $\Sigma$ be a diagram of affine type containing a subdiagram of type $v_1\stackrel{4}\longrightarrow v_2$ with $d_1=1,d_2=4$. A coefficient vector $\b$ is admissible if and only if $b_1=-2b_2\le 0$.
\end{theorem}

We will abuse notation by calling the condition in  Theorem~\ref{t_af_double4} the {\em annulus property} as well.  

\begin{proof}[Proof of Theorem~\ref{t_af_double4}]

The necessity follows from Theorem~\ref{t_rk2_}, part (2).

To prove the sufficiency, notice that all diagrams in question correspond to an unpunctured disk with two orbifold points and several marked points at the boundary (see~\cite{FeSTu3}). In particular, any triangulation corresponding to such a diagram consists of a monogon shown in Fig.~\ref{orb}(d) and a polygon $S_{out}$.

Now the proof is similar to the part (2) of the rank 2 case. 
In the assumptions of the theorem, the square roots of the weights of the diagram change under mutations in the same way as the weights of arrows of the diagram of type $\t C_n$ shown in Fig.~\ref{af_non}  with coefficient vector satisfying the annulus property, so the statement follows from Theorem~\ref{t_af_double}. 
   \end{proof}

\subsection{Extended affine mutation classes}

The result here is similar to the skew-sym\-metric case.

\begin{theorem}
There are no admissible vectors for diagrams of any extended affine mutation class.
\end{theorem}

\begin{proof}
Assume that $\Sigma$ is one of the four diagrams of extended affine type (see Fig.~\ref{fg}), and let $\bm b$ be an admissible coefficient vector. It is clear that $\bm b$ must satisfy the annulus property. 

  The diagrams  of types $F_4^{(*,*)}$ and  $G_2^{(*,+)}$ have unfoldings to quivers of type $E_6^{(1,1)}$, the diagram of type $F_4^{(*,+)}$ has an unfolding to a quiver of type $E_7^{(1,1)}$, and the diagram of type $G_2^{(*,*)}$ has an unfolding to a quiver of type $E_8^{(1,1)}$, where all the unfolded quivers are precisely those shown in Figs.~\ref{e6_11},~\ref{e7_11}, and~\ref{e8_11}, see~\cite{FeSTu2}. Moreover, it is easy to see that the mutation sequences used in the proof of Theorems~\ref{t_e6_11} and~\ref{t_e8_11} are sequences of composite mutations  (with respect to the unfoldings above) for certain sequences of mutations for the diagrams.

Therefore, if we take an unfolding $Q$ of $\Sigma$ with the unfolded coefficient vector $\h{\bm b}$ and apply a mutation sequence $\h\mu$ constructed in Section~\ref{extended} such that $\h\mu(\h{\bm b})$ does not satisfy the annulus property, then there exists a mutation sequence $\mu$ of $\Sigma$ such that $\mu(\bm b)$ does not satisfy the annulus property either, which shows that $\bm b$ cannot be admissible.
\end{proof}

\subsection{Diagrams from orbifolds}
We now want to extend Theorem~\ref{t_standard} to the orbifolds case. As in the surface case, we exclude finite and affine types, i.e. unpunctured disks with at most two orbifold points and once punctured disks with at most one orbifold point. 

First, we can define a standard triangulation of an orbifold in a similar way. We add orbifold points to the list of features, and place them to the left of all other features. We then place the two leftmost orbifold points in a monogon (see Fig.~\ref{orb}(d)), and all the others in digons (Fig.~\ref{orb}(a)), the subdiagram corresponding to the digon is shown in Fig.~\ref{orb}(b).

Vertices $v_1$ and $v_2$ are defined in the same way as in the surface case. In the case of orbifold points being the only features (note that there should be at least three of them and thus there is at least one digon, otherwise the diagram is of finite or affine type),  $v_1$ and $v_2$ are defined as in Fig.~\ref{orb}(c). The set $I$ is also defined in the same way.

\begin{figure}[!h]
\begin{center}
  \epsfig{file=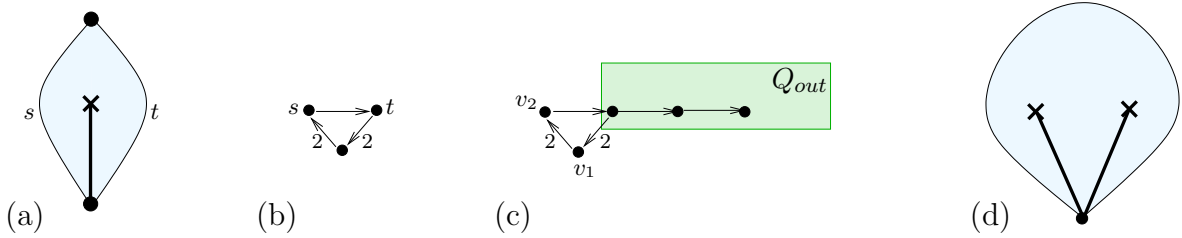,width=0.95\linewidth}
\put(-445,0){(a)}  
\put(-438,40){\scriptsize $s$}  
\put(-390,40){\scriptsize $t$}  
\put(-350,0){(b)}  
\put(-338,42){\scriptsize $s$}  
\put(-301,42){\scriptsize $t$}  
\put(-329,29){\tiny $2$}  
\put(-310,29){\tiny $2$}  
\put(-260,0){(c)}  
\put(-241,29){\tiny $2$}  
\put(-220,29){\tiny $2$}  
\put(-230,19){\scriptsize $v_1$}  
\put(-252,45){\scriptsize $v_2$}  
\put(-155,52){\small $Q_{out}$}  
\put(-80,0){(d)}  
\caption{Standard triangulations of orbifolds: (a) a triangulated digon with an orbifold point, (b) corresponding quiver, (c) vertices $v_1$ and $v_2$, (d) triangulated monogon with two orbifold points. }
\label{orb}
\end{center}
\end{figure}

\begin{theorem}
  \label{t_standard_orb}
  Let  $\O$ be an orbifold with at least one boundary component distinct from an unpunctured disk with at most two orbifold points and from once punctured disk with at most one orbifold point.
  Suppose that $\O$ is triangulated in the standard way.
  Then a coefficient vector $\b=(b_1,\dots,b_n)$ is admissible if and only if it satisfies the following conditions:
 \begin{itemize} 
\item[(a1)] $b_i=0$ for  $i\in I$; 
\item[(a2)] the annulus property is satisfied;
\item[(a3)] for the vertices $v_1$ and $v_2$ one has $b_1=-b_2\le 0$ if $d_1\ge d_2$, and $b_1=-2b_2\le 0$ if $d_1< d_2$.  
  
 \end{itemize}  
\end{theorem}

The proof is exactly the same as in the surface case. The only extra case is when the only features are orbifold points (otherwise, all the arcs incident to orbifold points belong to the set $I$), and, as in the surface case, the admissibility condition is prescribed by the shear coordinates $-b_1$ and $-b_2$ of the closed curve $C$. If $d_1=2d_2$, then  one has $b_1=-1=-b_2$, and if $d_2=2d_1$, then one has  $b_1=-2$ and $b_2=1$, which gives precisely (a3).

\section{Annulus property as criterion of mutation finiteness}
\label{sec_cr}

We now prove a criterion of mutation finiteness in terms of the annulus property applied to the whole mutation class, this was proposed by Sergey Fomin. For simplicity, we consider the skew-symmetric case (i.e., quivers), which then can be easily generalized to the skew-symmmetrizable case (see Remark~\ref{gen}).

\begin{theorem}
 \label{cr3} 
  Let $Q$ be a quiver with a frozen vertex $v$. Suppose that the subquiver $Q\setminus v$ is mutation-finite. 
 Then $Q$ is mutation-finite if and only if the the annulus property holds in every quiver $Q'$ mutation-equivalent to $Q$ for every double arrow contained in $Q'\setminus v$.
  
\end{theorem}  

\begin{proof}
  The necessity follows from Corollary~\ref{cor-ann}. Below we prove that the condition is also sufficient. As before, we denote by $\b$ the coefficient vector associated with the vertex $v$.

  Assume that the annulus property holds in any quiver of the mutation class. Let $Q_*=\langle Q\setminus v\rangle$ be the subquiver spanned by all mutable vertices. We will consider the following cases:
  \begin{itemize}
  \item[--] if $Q_*$ is of finite type, then every vector $\b$ is admissible by~\cite{FZ4};
  \item[--] if $Q_*$ is affine, then the statement follows from Theorem~\ref{aff_gen};% (and instead of checking all seeds it is enough to check any seed containing a double arrow); 
  \item[--]  if $Q_*$ is mutation-equivalent to $X_6$, then the statement follows from Remark~\ref{x6-ann};
  \item[--] if $Q_*$ is mutation-equivalent to $X_7$ or $E_6^{1,1}, E_7^{1,1}, E_8^{1,1}$, then the proofs of Theorems~\ref{t_e6_11}--\ref{t_x7} show that the annulus property cannot be satisfied in every quiver of the mutation class, and thus there are no admissible vectors;
  \item[--] if $Q_*$ is of rank 2, the statement follows immediately from Theorem~\ref{t_rk2};  

  \item[--] otherwise, $Q_*$ is a quiver arising from a surface; the rest of the proof below is aimed at settling the question 
    for this case.
  \end{itemize}

  From now on we will assume that $Q_*$ is of surface type, and will assume that $\b$ is {\bf not} admissible. Our aim is to find a quiver in the mutation class for which the annulus property does not hold. 

  By Theorem~\ref{t_per}, as $\b$ is not admissible, it corresponds to a non-peripheral lamination $L$. Hence, there exists a closed curve intersecting the lamination $L$ in a non-trivial way. We will pick such an intersecting curve in a particular way and will use it to construct a special triangulation, which will provide us with a quiver where annulus property does not hold.

  Let $g$ be the genus of the surface $S$ corresponding to the quiver $Q_*$. We will consider separately the cases of $g\ge 2$, $g=1$ and $g=0$.

  \medskip
  \noindent
   \underline{$\mathbf{g\ge 2}$}. For a surface $S$ of genus $g\ge 2$, the pure mapping class group $PMod(S)$ is generated by finitely many Dehn twists with respect to non-separating curves (see e.g.~\cite[Corollary~4.16]{FaM}). One can choose these curves as in Fig.~\ref{fig-gen}, left (these are called {\em Humphries generators}, see e.g.~\cite[Section 4.4.4]{FaM}).  Denote these curves $C_1,\dots, C_m$.
  As the lamination $L$ is not peripheral, at least one of the curves  $C_1,\dots, C_m$ intersects $L$ (otherwise, the pure mapping class group will act on $L$ trivially). We will assume that $C_1$ intersects $L$.

\begin{figure}[!h]
  \begin{center}
    \epsfig{file=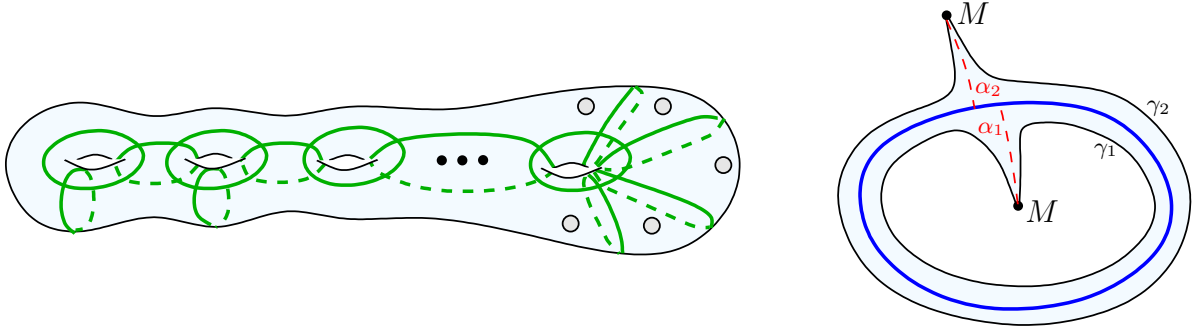,width=0.99\linewidth}
\put(-82,74){\color{red} \scriptsize $\alpha_1$}
\put(-83,89){\color{red} \scriptsize $\alpha_2$}
\put(-38,66){\scriptsize $\gamma_1$}
\put(-18,82){\scriptsize $\gamma_2$}
\put(-64,40){$M$}
\put(-90,115){$M$}    
\caption{Left: Humphries generators of $PMod$ in case of $g\ge 2$ (small circles stay for punctures and boundary components). Right: construction of the annulus. }
\label{fig-gen}
\end{center}
\end{figure}

  Let $M$ be a marked point (either a boundary marked point of a puncture). Let $\alpha_1$ and $\alpha_2$ be non-intersecting non-self-intersecting paths connecting $M$ to each of the sides of $C_1$, such paths exist since $C_1$ is a non-separating curve. Let $\gamma_1$ and $\gamma_2$ be loops based at $M$ and constructed by $\gamma_1=\alpha_1\beta\alpha_1^{-1}$ and  by $\gamma_2=\alpha_2\beta\alpha_2^{-1}$, where $\beta$ is the closed path along $C_1$, see Fig.~\ref{fig-gen}, right.
  Let $T$ be any triangulation containing arcs $\gamma_1$ and $\gamma_2$. We will show that in the triangulation $T$ the annulus property breaks.

  Indeed, as $L$ crosses $C_1$ non-trivially, the restriction of $L$ to the annulus bounded by  $\gamma_1$ and $\gamma_2$ is a non-empty  non-peripheral curve. Hence, the annulus property for this annulus does not hold in view of Lemma~\ref{l_ann}.
  
  \medskip
  \noindent
  \underline{$\mathbf{g=1}$}. In this case, one cannot always find the generators of the mapping class group by twists along non-separating curves only, however there exists a set of generators by twists along finitely many non-separating curves and $k-1$ boundary curves, where $k$ is the number of boundary components on $S$ (here, by a boundary curve we mean a closed peripheral curve along one of the boundary components), see~\cite{FaM}.

  If $k\le 1$, then there still exists a set of generators by twists in non-separating curves, and we can use the same reasoning as before.

  If $k>1$, then for every generator we can find at least one marked point on each side with respect to the corresponding curve (i.e. for non-separating curves we proceed as before, and for a boundary curve we choose a marked point lying on the corresponding boundary component and a marked point  lying on a different boundary component). So, we still are able to construct the annulus as in Fig.~\ref{fig-gen}, right, but possibly using different marked points on different sides.

    \medskip
  \noindent
   \underline{$\mathbf{g=0}$}. In this case, any generator of the pure mapping class group is a twist along a separating curve, but for any such curve one can find at least one marked point on each of its sides, and thus one can apply the same construction of an annulus as before.

 \end{proof}

 \begin{cor}
   \label{cor_cr}
Let $Q$ be a quiver with a frozen vertex. Then
 $Q$ is mutation-finite if and only for every quiver $Q'$ in the mutation class of $Q$ every rank $3$ subquiver of $Q'$ is mutation-finite.

 \end{cor}  

 \begin{proof}
   The ``only if'' direction is evident. Suppose that $Q$ is mutation-infinite. If the subquiver $Q_*=Q\setminus v$ (where $v$ is the frozen vertex) is also mutation-infinite, then the mutation class of $Q_*$ contains a quiver with an arrow of multiplicity higher than $2$ and hence any connected rank $3$ subquiver containing that arrow is mutation-infinite. If $Q_*$ is mutation-finite, then, by Theorem~\ref{cr3},  
   there is a quiver $Q'$ in the mutation class of $Q$ where the annulus property does not hold (for some vertices $v_1, v_2$ connected by a double arrow). Then the rank 3 subquiver of $Q'$ spanned by $\langle v_1,v_2,v\rangle$ is mutation-infinite.
   
 \end{proof}  
 
 \begin{remark}
   \label{gen}
Both Theorem~\ref{cr3}  and Corollary~\ref{cor_cr} can be easily generalized to the skew-symmetrizable case,
where the annulus property is understood as in Theorem~\ref{t_af_double4}. The proofs are the same as in the skew-symmetric case. %  (for an orbifold, the modular group is generated by the same curves as in the case of the surface obtained when substituting all orbifold points by puncture). 

\end{remark}

\end{document}